\documentclass[twoside,leqno,10pt, A4]{amsart}
\usepackage{amsfonts}
\usepackage{amsmath}
\usepackage{amscd}
\usepackage{amssymb}
\usepackage{amsthm}
\usepackage{amsrefs}
\usepackage{latexsym}
\usepackage{mathrsfs}
\usepackage{bbm}
\usepackage{amscd}
\usepackage{amssymb}
\usepackage{amsthm}
\usepackage{amsrefs}
\usepackage{latexsym}
\usepackage{mathrsfs}
\usepackage{bbm}
\usepackage{enumerate}
\usepackage{graphicx}
\usepackage{color}
\begin{document}

\newtheorem{theorem}[subsection]{Theorem}
\newtheorem{proposition}[subsection]{Proposition}
\newtheorem{lemma}[subsection]{Lemma}
\newtheorem{corollary}[subsection]{Corollary}
\newtheorem{conjecture}[subsection]{Conjecture}
\newtheorem{prop}[subsection]{Proposition}
\numberwithin{equation}{section}
\newcommand{\mr}{\ensuremath{\mathbb R}}
\newcommand{\mc}{\ensuremath{\mathbb C}}
\newcommand{\dif}{\mathrm{d}}
\newcommand{\intz}{\mathbb{Z}}
\newcommand{\ratq}{\mathbb{Q}}
\newcommand{\natn}{\mathbb{N}}
\newcommand{\comc}{\mathbb{C}}
\newcommand{\rear}{\mathbb{R}}
\newcommand{\prip}{\mathbb{P}}
\newcommand{\uph}{\mathbb{H}}
\newcommand{\fief}{\mathbb{F}}
\newcommand{\majorarc}{\mathfrak{M}}
\newcommand{\minorarc}{\mathfrak{m}}
\newcommand{\sings}{\mathfrak{S}}
\newcommand{\fA}{\ensuremath{\mathfrak A}}
\newcommand{\mn}{\ensuremath{\mathbb N}}
\newcommand{\mq}{\ensuremath{\mathbb Q}}
\newcommand{\half}{\tfrac{1}{2}}
\newcommand{\f}{f\times \chi}
\newcommand{\summ}{\mathop{{\sum}^{\star}}}
\newcommand{\chiq}{\chi \bmod q}
\newcommand{\chidb}{\chi \bmod db}
\newcommand{\chid}{\chi \bmod d}
\newcommand{\sym}{\text{sym}^2}
\newcommand{\hhalf}{\tfrac{1}{2}}
\newcommand{\sumstar}{\sideset{}{^*}\sum}
\newcommand{\sumprime}{\sideset{}{'}\sum}
\newcommand{\sumprimeprime}{\sideset{}{''}\sum}
\newcommand{\shortmod}{\ensuremath{\negthickspace \negthickspace \negthickspace \pmod}}
\newcommand{\V}{V\left(\frac{nm}{q^2}\right)}
\newcommand{\sumi}{\mathop{{\sum}^{\dagger}}}
\newcommand{\mz}{\ensuremath{\mathbb Z}}
\newcommand{\leg}[2]{\left(\frac{#1}{#2}\right)}
\newcommand{\muK}{\mu_{\omega}}
\newcommand{\thalf}{\tfrac12}
\newcommand{\lp}{\left(}
\newcommand{\rp}{\right)}

\theoremstyle{plain}
\newtheorem{conj}{Conjecture}
\newtheorem{remark}[subsection]{Remark}

\makeatletter
\def\widebreve{\mathpalette\wide@breve}
\def\wide@breve#1#2{\sbox\z@{$#1#2$}%
     \mathop{\vbox{\m@th\ialign{##\crcr
\kern0.08em\brevefill#1{0.8\wd\z@}\crcr\noalign{\nointerlineskip}%
                    $\hss#1#2\hss$\crcr}}}\limits}
\def\brevefill#1#2{$\m@th\sbox\tw@{$#1($}%
  \hss\resizebox{#2}{\wd\tw@}{\rotatebox[origin=c]{90}{\upshape(}}\hss$}
\makeatletter

\title[Moments and Non-vanishing of central values of Quadratic {H}ecke $L$-functions in the Gaussian Field]{Moments and Non-vanishing of central values of Quadratic Hecke $L$-functions in the Gaussian Field}

\author{Peng Gao}
\address{School of Mathematical Sciences, Beihang University, Beijing 100191, China}
\email{penggao@buaa.edu.cn}
\begin{abstract}
 We evaluate the first three moments of central values of a family of qudratic Hecke $L$-functions in the Gaussian field with power saving error terms. In particular, we obtain asymptotic formulas for the first two moments with error terms of size $O(X^{1/2+\varepsilon})$. We also study the first and second mollified moments of the same family of $L$-functions to show that at least $87.5\%$ of the members of this family have non-vanishing central values.

\end{abstract}

\maketitle



\section{Introduction}
\label{sec 1}

The study on moments of quadratic twists of $L$-functions at central values has important applications to problems such as class numbers of imaginary quadratic fields, ranks of elliptic curves and the existence of Landau-Siegel zeros. For the central values of the family of quadratic Dirichlet $L$-functions, M. Jutila evaluated the first two moments in \cite{Jutila} to show that there are infinitely many $L$-functions in this family with non-vanishing central values. This approach was further advanced by  K. Soundararajan in \cite{sound1}, who computed the first and second mollified moments of the family of primitive quadratic Dirichlet $L$-functions to show that at least $87.5\%$ of such $L$-functions have non-vanishing central values.

  In the same paper \cite{sound1}, Soundrarajan also obtained the third moment of the family of primitive quadratic Dirichlet $L$-functions. Under the assumption of the Generalized Riemann Hypothesis (GRH), Q. Shen obtained an asymptotic formula in \cite{Shen} for the fourth moment of the same family.
Analogue to quadratic twists by Dirichlet characters,  there is also an intensive study on moments of various families of modular forms.
Results on the first moments can be found in \cites{Munshi1, Petrow1}. Assuming GRH, the second moment of quadratic twists of modular $L$-functions was computed by K. Soundararajan and M. P. Young in \cite{S&Y}.

 Other than obtaining the main terms of the moments of families of $L$-functions, a lot of attention has been drawn upon the improvement of the error terms.
 For the first moment of the family of quadratic Dirichlet $L$-functions,  the error term obtained in Jutila's result is of size $O(X^{3/4+\varepsilon})$ (with the main term being about $X \log X$). An error term of the same size is obtained by A. I. Vinogradov and L. A. Takhtadzhyan \cite{ViTa} and was improved to $X^{19/32 + \varepsilon}$ by D. Goldfeld and J. Hoffstein in \cite{DoHo}. In fact, an error term of size $O(X^{1/2 + \varepsilon})$ is essentially implicit in \cite{DoHo} (see the remarks in the paragraph below Theorem 1.1 of \cite{Young1}).

  The result of Goldfeld and Hoffstein is obtained via the usage of Eisenstein series of metaplectic type. Using a different approach which involves with more classical tools from analytical number theory, M. P. Young in \cite{Young2} was able to establish the same estimation for the error term for a smoothed first moment. Young's approach builds on the previous work of Soundararajan, who developed a type of Poisson summation formula for smoothed quadratic Dirichlet
character sums. In the meanwhile, Young also introduced novel techniques such as using a recursive relation to lower down the error term successively as well as performing an intricate analysis of certain subsidiary terms whose sizes are difficult to control individually. These techniques have been successfully applied by Young later in \cite{Young2} to improve the error term in the smoothed third moment of the family of primitive quadratic Dirichlet $L$-functions and by K. Sono \cite{Sono} for the smoothed second moment of the same family.

  Inspired by the work of Soundararajan and Young, we expect to apply the methods in \cites{sound1, Young1, Young2} to study moments of other quadratic twists of $L$-functions. In this paper, we focus on the moments of a family of quadratic Hecke $L$-functions in the Gaussian field. Thus, we denote $K=\mq(i)$ for the Gaussian field throughout the paper and we denote $\mathcal{O}_K=\mz[i]$ for its ring of integers and $U_K=\{ \pm 1, \pm i \}$ for the group of units in $\mathcal{O}_K$. Recall that every ideal in $\mathcal{O}_K$ co-prime to $2$ has a unique generator congruent to $1$ modulo $(1+i)^3$ (see the definition above Lemma 8.2.1 in \cite{BEW}). These generators are called primary. We shall denote $\omega$ for a prime number in $\mathcal{O}_K$, by which we mean that the ideal $(\omega)$ generated by $\omega$ is a prime ideal. We write $N(n)$ for the norm of any $n \in K$.
We further denote $\chi$ for a Hecke character of $K$ and we say that $\chi$ is of trivial infinite type if its component at infinite places of $K$ is trivial.
We write $L(s,\chi)$ for the $L$-function associated to $\chi$ and we denote $\zeta_{K}(s)$ for the Dedekind zeta function of $K$.

 For any element $n \in \mathcal{O}_K$, we say $n$ is odd if $(n,2)=1$ and we say $n$ is square-free if the ideal $(n)$ is not divisible by the square of any prime ideals. We further denote $\chi_c=\leg {c}{\cdot}$, where $\leg {\cdot}{\cdot}$ is the quadratic residue symbol defined in Section \ref{sec2.4}.  Similar to the arguments in Section 2.1 of \cite{G&Zhao4}, the symbol $\chi_{(1+i)^5d}$ defines a primitive quadratic Hecke character modulo $(1+i)^5d$ of trivial infinite type when $d \in \mathcal{O}_K$ is odd and square-free.

  We can thus consider the moments of the family of quadratic Hecke $L$-functions $L(\thalf, \chi_{(1+i)^5d})$ with $d$ varying over odd and square-free elements in $\mathcal{O}_K$. The aim here is not only to obtain valid asymptotic formulas, but also to obtain error terms as good as those given by Young and Sono for the classical case.  Our first result is the following
\begin{theorem}
\label{theo:mainthm}
Let $\Phi:\mr^{+} \rightarrow \mr$ be a smooth function of compact support.  Then for $1 \leq j \leq 3$ and any $\varepsilon>0$, we have
\begin{equation}
\label{eq:mainthm}
 \sumstar_{(d,2)=1} L(\thalf, \chi_{(1+i)^5d})^j \Phi\leg{N(d)}{X} =  X P_i(\log{X}) + O(X^{\theta_i + \varepsilon}),
\end{equation}
for some polynomials $P_i$ of degree $i(i+1)/2$ (depending on $\Phi$) and $\theta_i=1/2$ for $i=1,2$, $\theta_3=3/4$. Here the ``$*$'' on the sum over $d$ means that the sum is restricted to square-free elements $d$ in $\mathcal{O}_K$.
\end{theorem}

   In order to establish Theorem \ref{theo:mainthm}, we shall use recursive arguments to obtain the desired error terms in \eqref{eq:mainthm}, starting from  much larger error terms. This process actually requires us to consider a more general situation, namely the following ``twisted'' moments for primary $l \in \mathcal{O}_K$:
\begin{align}
\begin{split}
\label{Malphal}
  M_{\alpha}(l) =& \sumstar_{(d,2)=1}  L(\thalf + \alpha, \chi_{(1+i)^5d})\chi_{(1+i)^5d}(l) \Phi\left(\frac{N(d)}{X}\right), \\
  M_{\alpha, \beta}(l) =& \sumstar_{(d,2)=1} L(\thalf + \alpha, \chi_{(1+i)^5d}) L(\thalf + \beta, \chi_{(1+i)^5d})\chi_{(1+i)^5d}(l)\Phi\left(\frac{N(d)}{X}\right), \\
  M_{\alpha, \beta, \gamma}(l) =&\sumstar_{(d,2)=1} L(\thalf + \alpha, \chi_{(1+i)^5d}) L(\thalf + \beta, \chi_{(1+i)^5d}) L(\thalf + \gamma, \chi_{(1+i)^5d}) \chi_{(1+i)^5d}(l)
\Phi\left(\frac{N(d)}{X}\right).
\end{split}
\end{align}

  It is certainly expected that the case $i=1$ of \eqref{eq:mainthm} is the easiest to study compared to higher moments. In fact, we shall only need to evaluate $M_{\alpha}(l)$ for $l$ being square-free while for higher moments, we need to evaluate $M_{\alpha, \beta}(l)$ and $M_{\alpha, \beta, \gamma}(l)$ for a general $l$. In order to state our results concerning $M_{\alpha}(l), M_{\alpha, \beta}(l)$ and $M_{\alpha, \beta, \gamma}(l)$, we need to first introduce a few notations. For $\Phi$ given as in the statement of Theorem \ref{theo:mainthm}, we shall set
\begin{align*}
     F(x)=\Phi(\frac {x}{X}).
\end{align*}
   We recall that the Mellin transform $\hat{f}$ for any function $f$ is defined to be
\begin{align*}
     \widehat{f}(s) =\int\limits^{\infty}_0f(t)t^s\frac {\dif t}{t}.
\end{align*}
  It follows from this that we have
\begin{align*}
     \widehat{F}(s)=X^s\widehat{\Phi}(s).
\end{align*}

    For a sequence of complex numbers
$\alpha_1, \cdots, \alpha_j$ and a primary $n \in \mathcal{O}_K$, we define
\begin{align}
\label{eq:sigma}
 \sigma_{\alpha_1, \cdots, \alpha_j}(n) = \sum_{\substack{ a_1\cdots a_j=n \\ a_i \equiv 1 \bmod {(1+i)^3}, 1 \leq i \leq j } }
\prod^j_{i=1}N(a_i)^{-\alpha_i}.
\end{align}

  For any primary $l \in \mathcal{O}_K$, we shall use $l_1, l_2$ for the unique primary elements in $\mathcal{O}_K$ such that $l=l_1l_2$ with $l_1$ being square-free and $l_2$ a square. We shall use the notation $l^*$ for $l_1$ as well. Using this notation, we define for each $l$,
\begin{align*}
 A_{\alpha_1, \cdots, \alpha_j}(l) = \sum_{\substack{n \equiv 1 \bmod {(1+i)^3} \\(n,2)=1}} \frac{\sigma_{\alpha_1, \cdots, \alpha_j}(l^* n^2)}{N(n)}
\prod_{\substack{\varpi \equiv 1 \bmod {(1+i)^3} \\ \varpi | n l}} (1 + N(\varpi)^{-1})^{-1}.
\end{align*}

  We further define $B_{\alpha}(l), B_{\alpha,\beta}(l)$ and $B_{\alpha, \beta, \gamma}(l)$ such that
\begin{align}
\label{B}
\begin{split}
 A_{\alpha}(l) = & \zeta_{K,2}(1 + 2\alpha)B_{\alpha}(l), \\
 A_{\alpha, \beta}(l) = & \zeta_{K,2}(1 + 2\alpha) \zeta_{K,2}(1 + 2\beta)  \zeta_{K,2}(1 + \alpha + \beta)B_{\alpha,\beta}(l), \\
 A_{\alpha, \beta, \gamma}(l) = & \zeta_{K,2}(1 + 2\alpha) \zeta_{K,2}(1 + 2\beta) \zeta_{K,2}(1 + 2\gamma) \zeta_{K,2}(1 + \alpha + \beta)
\zeta_{K,2}(1 + \alpha + \gamma) \zeta_{K,2}(1 + \beta + \gamma) B_{\alpha,\beta,\gamma}(l),
\end{split}
\end{align}
   where we define the function $\zeta_{K,l}(s)$ for $l \in \mathcal{O}_K$ by removing the Euler factors from $\zeta_{K}(s)$ at prime ideals $(\varpi)$ with $\varpi | l$. We define similarly $L_{l}(s, \chi)$ for any Hecke character $\chi$ of $K$, so that
\begin{align*}
     L_l(s, \chi)=L(s,\chi)\prod_{\substack{ (\varpi) \\ \varpi | l}}\Big(1-\frac {\chi(\varpi)}{N(\varpi)^s} \Big ).
\end{align*}

   We note here (see also the discussions below Lemma \ref{lemma:A}) that $B_{\alpha}(l), B_{\alpha,\beta}(l)$ and $B_{\alpha, \beta, \gamma}(l)$ have absolutely convergent Euler products for the parameters $\alpha, \beta, \gamma$ in a neighborhood of the origin. For example, we have
\begin{align*}
 B_{\alpha}(l) = N(l^*)^{-\alpha/2}\prod_{\substack{ \varpi \equiv 1 \bmod {(1+i)^3} \\ \varpi | l}}  (1 + N(\varpi)^{-1})^{-1}
\prod_{\substack{ \varpi \equiv 1 \bmod {(1+i)^3} \\ \varpi \nmid 2l}} \lp 1 - N(\varpi)^{-2-2\alpha} (1+N(\varpi)^{-1})^{-1} \rp.
\end{align*}

   We define further that
\begin{align}
\label{gamma}
 \Gamma_{\alpha} =&  \left(\frac{32}{\pi^2} \right)^{-\alpha}
\frac{\Gamma\left(\thalf- \alpha \right)}{\Gamma\left(\thalf + \alpha \right)}.
\end{align}

Now, we are ready to state our recursive results concerning the error terms for the asymptotic expressions of $M_{\alpha}(l)$, $M_{\alpha, \beta}(l)$ and  $M_{\alpha, \beta, \gamma}(l)$.
\begin{theorem}
\label{theo:recursive2}
  Suppose that for any primary $l \in \mathcal{O}_K$, we write $l=l_1l_2$ such that both $l_1, l_2$ are primary and that $l_1$ is square-free and $l_2$ is a square. If we have uniformly for $\alpha, \beta, \gamma$ lying in the rectangle $|\Re(s)| \leq \frac{\varepsilon}{\log{X}}$, $|\Im(s)| \leq X^{\varepsilon}$ that
\begin{align}
\label{eq:Malphal}
M_{\alpha}(l)
=& \pi \sum_{\epsilon_1 \in \{\pm 1 \}}
A_{\epsilon_1 \alpha}(l) \Gamma_{\alpha}^{\delta_1}
\frac{ \widehat{F}(1- \delta \alpha)}{2 \zeta_{K,2}(2) \sqrt{N(l)}} + O(X^{f} \sqrt{N(l)} (N(l) X)^{\varepsilon}) \quad \text{for $l$ square-free},  \\
\label{eq:Malpha2}
 M_{\alpha, \beta}(l)
=&\pi \sum_{\epsilon_1, \epsilon_2 \in \{\pm 1 \}}
A_{\epsilon_1 \alpha, \epsilon_2 \beta }(l) \Gamma_{\alpha, \beta}^{\delta_1, \delta_2}
\frac{ \widehat{F}(1- \delta_1 \alpha - \delta_2 \beta)}{2 \zeta_{K,2}(2) \sqrt{N(l^*)}}  + O(X^{f} \sqrt{N(l)} (N(l) X)^{\varepsilon}), \\
\label{eq:Malpha3}
 M_{\alpha, \beta, \gamma}(l)
=&\pi \sum_{\epsilon_1, \epsilon_2, \epsilon_3 \in \{\pm 1 \}}
A_{\epsilon_1 \alpha, \epsilon_2 \beta,\epsilon_3 \gamma}(l) \Gamma_{\alpha, \beta, \gamma}^{\delta_1, \delta_2, \delta_3}
\frac{ \widehat{F}(1- \delta_1 \alpha - \delta_2 \beta - \delta_3 \gamma)}{2 \zeta_{K,2}(2) \sqrt{N(l^*)}}  + O(X^{f} \sqrt{N(l)} (N(l) X)^{\varepsilon}),
\end{align}
 for $f>1/2$ in \eqref{eq:Malphal}, \eqref{eq:Malpha2} and for $f>3/4$ in \eqref{eq:Malpha3}, then the expression \eqref{eq:Malphal} holds for $f$ replaced by $\frac 14+\frac{f}{2}$, the expression \eqref{eq:Malpha2} holds for $f$ replaced by $1-\frac{1}{4f}$ and the expression \eqref{eq:Malpha3} holds for $f$ replaced by $\frac34 + \frac{f-\frac34}{2f}$. Here, we define $\Gamma_{\alpha,\beta,\gamma}^{\delta_1, \delta_2, \delta_3}$ to be $\Gamma_{\alpha}^{\delta_1} \Gamma_{\beta}^{\delta_2} \Gamma_{\gamma}^{\delta_3}$, where $\delta_i = 0$ if $\epsilon_i = +1$, and $\delta_i = 1$ if $\epsilon_i = -1$. Similar definitions  apply to $\Gamma_{\alpha}^{\delta_1}$ and $\Gamma_{\alpha, \beta}^{\delta_1, \delta_2}$.
\end{theorem}

  We note here that our condition in Theorem \ref{theo:recursive2} for $M_{\alpha}(l)$ is slightly different compared to those for $M_{\alpha, \beta}(l)$ and
$M_{\alpha, \beta, \gamma}(l)$. This is because that we only need $l$ to be square-free in the proof for the case of $M_{\alpha}(l)$ while for the other cases, a general $l$ is involved. We also note that in \cite{CFKRS}, J. B. Conrey, D. Farmer, J. Keating, M. Rubinstein and N. Snaith produced a recipe that allows one to conjecture the asymptotics for the integral moments of families of $L$-functions. Modifying their recipe, one may obtain conjecturally the main terms  for $M_{\alpha}(l), M_{\alpha, \beta}(l)$ and  $M_{\alpha, \beta, \gamma}(l)$ given in \eqref{eq:Malphal}-\eqref{eq:Malpha3}, as Young and Sono did in \cites{Young1, Young2, Sono} for the case of Dirichlet $L$-functions. We can also obtain the the same main terms here by going directly through the arguments in the proof of Theorem \ref{theo:recursive2} in the paper.

   Applying the convexity bound that (see \cite[Exercise 3, p. 100]{iwakow}) for $\Re(s) =\varepsilon$,
\begin{align*}
  L(1/2 + s, \chi_{(1+i)^5d}) \ll ((1+|s|)^2N(d))^{1/4+\varepsilon},
\end{align*}
  we deduce that expressions \eqref{eq:Malphal}-\eqref{eq:Malpha3} are valid for $f=1+j/4$ as an initial estimate. Arguing
as the proof of Conjecture 3.3 in \cite{Young1}, we see that this leads to a valid expression of \eqref{eq:Malphal} and \eqref{eq:Malpha2} for $f=\half$, as well as a valid expression of \eqref{eq:Malpha3} for $f=\frac 34$. We summarize this in the following result.
\begin{theorem}
\label{thm: Malphal1}
   Let $M_{\alpha}(l), M_{\alpha, \beta}(l)$ and  $M_{\alpha, \beta, \gamma}(l)$ given in \eqref{eq:Malphal}-\eqref{eq:Malpha3}.  For any primary $l \in \mathcal{O}_K$ and any complex number $\alpha, \beta, \gamma$ lying in the rectangle $|\Re(s)| \leq \frac{\varepsilon}{\log{X}}$, $|\Im(s)| \leq X^{\varepsilon}$, the expression \eqref{eq:Malphal} holds with an error of size $N(l)^{1/2 + \varepsilon} X^{\frac{1}{2} + \varepsilon}$ for any $\varepsilon>0$, when $l$ is square-free. For general $l$,  the expression \eqref{eq:Malpha2} holds with an error of size $N(l)^{1/2 + \varepsilon} X^{\frac{1}{2} + \varepsilon}$ for any $\varepsilon>0$ and the expression \eqref{eq:Malpha3} holds with an error of size $N(l)^{3/4 + \varepsilon} X^{\frac{1}{2} + \varepsilon}$ for any $\varepsilon>0$.
\end{theorem}

   In Section \ref{sec: pfrecursive}, we shall prove Theorem \ref{theo:recursive2} by assuming that each $\alpha, \beta, \gamma$ lies in a punctured rectangle of the form $|\Re(s)| \leq c_1/\log{X}$, $|\Im(s)| \leq c_2 X^{\varepsilon}$ minus $|\Re(s)| \leq c_1/(2\log{X})$, $|\Im(s)| \leq (c_2/2) X^{\varepsilon}$ for suitable
$c_i$ depending on $\alpha, \beta, \gamma$ such that the distances between the parameters are at least $\gg 1/X^{\varepsilon}$.
One then deduces the result for other cases following the arguments made in the paragraph above Section 3.3 in \cite{Young1} and the two paragraphs below Lemma  3.6 in \cite{Young2}.
By considering the limit case of $\alpha, \beta, \gamma \rightarrow 0, l=1$ in Theorem \ref{thm: Malphal1}, we recover the statement of Theorem \ref{theo:mainthm}. Note that we do not run into singularities here, see Lemma 2.3 in \cite{Sono} and the paragraph above it for an explanation.

   Note that the ``twisted'' moments $M_{\alpha}(l), M_{\alpha, \beta}(l)$ and $M_{\alpha, \beta, \gamma}(l)$ appear naturally when mollifying central values. Thus, our result in Theorem \ref{thm: Malphal1} also paves a way for us to consider the mollified moments of the same family of $L$-functions. We shall in fact evaluate the first and second mollified moments of this family in Section \ref{sect: nonvanishing} to establish the following non-vanishing result on central values.
\begin{theorem}
\label{thm: nonvanishing}
  We have for all large $x$, and any fixed $\varepsilon > 0$,
\begin{align*}
 & \sum_{\substack{N(d) \leq x \\ (d, 2)=1 \\ L(\half, \chi_{(1+i)^5d}) \neq 0}}\mu_{[i]}(d)^2 \geq
\lp \frac 78-\varepsilon \rp \sum_{\substack{N(d) \leq x \\ (d, 2)=1 }}\mu_{[i]}(d)^2.
\end{align*}
  Thus, for at least $87.5\%$ of the odd square-free elements $d \in \mathcal{O}_{K}$,  $L(\half,\chi_{(1+i)^5d})\neq 0$.
\end{theorem}

   Our proof of Theorem \ref{theo:recursive2} follows largely the line of treatment of Young in \cites{Young1, Young2}, as well as the approach of Sono in \cite{Sono} for the evaluation of $M_{\alpha, \beta}(l)$. We shall apply the approximate functional equation for
$L(s, \chi_{(1+i)^5d})$ obtained in Section \ref{sect: apprfcneqn} to express products involving $L(\half+\alpha, \chi_{(1+i)^5d})$ into two smoothed sums. Then we apply a two dimensional Poisson summation to convert the sum over $d$ into a dual sum. Shifting the contour of integrals leads to a contribution of poles, which in turn gives us two types of main terms, with the second type being contributed by non-zero squares in the dual sum.  On the new line of the integration, we apply the recursive argument to obtain ``tails" of these main terms, so that some of them combine naturally together. This leads to the main terms given in \eqref{eq:Malphal}-\eqref{eq:Malpha3} with desired error terms of smaller sizes.
The most intricate part of the above approach involves with representing the second type main terms so that they can be combined with certain terms coming from the recursive process. This requires a careful analysis on the Archimedean parts of functional equations of the corresponding  $L$-functions as well as the two dimensional Fourier transforms of the weight functions involved.

   On the other hand, our proof of Theorem \ref{thm: nonvanishing} owes much to the work of Soundararajan in \cite{sound1}. In fact, the error term in the asymptotic expression for $M_{\alpha, \beta}(l)$ given in Theorem \ref{thm: Malphal1} is not strong enough in the $l$ aspect for us to choose a mollifier that is long enough to derive our result. We need thus to follow the original treatment of Soundararajan in  \cite{sound1} to handle the second mollified moment. The proof of Theorem \ref{thm: nonvanishing} is then made much easier, thanks to the existing approach available in \cite{sound1}.

\section{Preliminaries}
\label{sec 2}

In this section, we include some auxiliary results needed in the proofs of our theorems.

\subsection{Quadratic residue symbol, Gauss sum and Poisson Summation}
\label{sec2.4}
   Recall that $K=\mq(i)$ and it is well-known that $K$ have class number one.
We denote $(D_K)$ for the discriminant of $K$ and recall that $D_{K}=-4$.
For $n \in \mathcal{O}_{K}, (n,2)=1$, we denote the symbol  $\leg {\cdot}{n}$ for the quadratic residue symbol $\pmod n$ in $K$. For a prime $\varpi \in \mz[i]$
with $N(\varpi) \neq 2$, the quadratic symbol is defined for $a \in
\mathcal{O}_{K}$, $(a, \varpi)=1$ by $\leg{a}{\varpi} \equiv
a^{(N(\varpi)-1)/2} \pmod{\varpi}$, with $\leg{a}{\varpi} \in \{
\pm 1 \}$. When $\varpi | a$, we define
$\leg{a}{\varpi} =0$.  Then the quadratic symbol  can be extended
to any composite $n$ with $(N(n), 2)=1$ multiplicatively. We further define $\leg {\cdot}{c}=1$ when $c \in U_K$.

  For any $n, r\in \mathcal{O}_K$, $(n,2)=1$, we define the quadratic Gauss sum $g(r, n)$  by
\begin{align}
\label{g2}
 g(r,n) = \sum_{x \bmod{n}} \leg{x}{n} \widetilde{e}\leg{rx}{n},
\end{align}
  where
\begin{align*}
 \widetilde{e}(z) =\exp \left( 2\pi i  \left( \frac {z}{2i} - \frac {\bar{z}}{2i} \right) \right) .
\end{align*}

  Let $\varphi_{[i]}(n)$ denote the number of elements in the reduced residue class of $\mathcal{O}_K/(n)$,
we now recall from \cite[Lemma 2.2]{G&Zhao4} some explicitly evaluations of $g(r,n)$ for $n$ being primary.
\begin{lemma} \label{Gausssum}
\begin{enumerate}[(i)]
\item  We have
\begin{align*}
 g(rs,n) & = \overline{\leg{s}{n}} g(r,n), \qquad (s,n)=1, \\
   g(k,mn) & = g(k,m)g(k,n),   \qquad  m,n \text{ primary and } (m , n)=1 .
\end{align*}
\item Let $\varpi$ be a primary prime in $\mathcal{O}_K$. Suppose $\varpi^{h}$ is the largest power of $\varpi$ dividing $k$. (If $k = 0$ then set $h = \infty$.) Then for $l \geq 1$,
\begin{align*}
g(k, \varpi^l)& =\begin{cases}
    0 \qquad & \text{if} \qquad l \leq h \qquad \text{is odd},\\
    \varphi_{[i]}(\varpi^l) \qquad & \text{if} \qquad l \leq h \qquad \text{is even},\\
    -N(\varpi)^{l-1} & \text{if} \qquad l= h+1 \qquad \text{is even},\\
    \leg {ik\varpi^{-h}}{\varpi}N(\varpi)^{l-1/2} \qquad & \text{if} \qquad l= h+1 \qquad \text{is odd},\\
    0, \qquad & \text{if} \qquad l \geq h+2.
\end{cases}
\end{align*}
\end{enumerate}
\end{lemma}

    We quote the following Poisson summation formula from \cite[Lemma 2.7]{G&Zhao4}.
\begin{lemma}
\label{Poissonsumformodd} Let $n \in \mz[i],   n \equiv 1 \pmod {(1+i)^3}$ and $\leg {\cdot}{n}$ be the quadratic residue symbol $\pmod {n}$. For any Schwartz class function $W$,  we have
\begin{align*}
   \sum_{\substack {m \in \mz[i] \\ (m,1+i)=1}}\leg {m}{n} W\left(\frac {N(m)}{X}\right)=\frac {X}{2N(n)}\leg {1+i}{n}\sum_{k \in
   \mz[i]}(-1)^{N(k)} g(k,n)\widetilde{W}\left(\sqrt{\frac {N(k)X}{2N(n)}}\right),
\end{align*}
  where
\begin{align*}
   \widetilde{W}(t) &=\int\limits^{\infty}_{-\infty}\int\limits^{\infty}_{-\infty}W(N(x+yi))\widetilde{e}\left(- t(x+yi)\right)\dif x \dif y, \quad t \geq 0.
\end{align*}
\end{lemma}

\subsection{Evaluation of certain integrals}
\label{sect: Wiestm}
     We will require an evaluation on $\widetilde{W}(t)$ for special choices of $W(t)$. First note that $\widetilde{W}(t) \in \mr$ in general for any $t \geq 0$, since we have
\begin{align}
\label{Wt}
     \widetilde{W}(t)  =\int\limits_{\mr^2}\cos (2\pi t y)W(x^2+y^2) \ \dif x \dif y.
\end{align}

     We evaluate the above integral in polar coordinates to get
\begin{align*}
     \widetilde{W}(t) =& 4\int\limits^{\pi/2}_0\int\limits^{\infty}_0\cos (2\pi t r\sin \theta)W(r^2) \ r \dif r \dif \theta    = 2\int\limits^{\pi/2}_0 \int\limits^{\infty}_0\cos (2\pi t r^{1/2}\sin \theta)W(r) \ \dif r \dif \theta.
\end{align*}

     We now take $\Phi(t)$ as given in Theorem \ref{theo:mainthm}. Fix a positive integer $m$,  we let $G_j(s), 1 \leq j \leq m$ be entire, even functions, bounded in any strip $-A \leq \Re(s) \leq A$ for some $A>2$ such that $G_j(0)=1, 1\leq j \leq m$. We further let $\alpha_j, 1 \leq j \leq m$ be complex numbers and denote $(\alpha_j)$ for the sequence $(\alpha_1, \cdots, \alpha_j)$.  We define further for $t>0$,
\begin{align}
\label{eq:Vdef}
 V_{(\alpha_j)}(t) = \frac{1}{2 \pi i} \int\limits\limits_{(2)} \frac{G_j(s)}{s} g_{(\alpha_j)}(s) t^{-s} ds,
\end{align}
 where $g_{(\alpha_j)}(s) =\prod^j_{i=1} g_{\alpha_i}(s)$ with
\begin{equation*}
 g_{\alpha}(s) = \left(\frac{2^{5/2}}{\pi}\right)^{s}
\frac {\Gamma(\frac{1}{2}+\alpha+s)}{\Gamma(\frac{1}{2}+\alpha)}.
\end{equation*}

  The functions $V_{(\alpha_j)}(t)$ appear naturally in the approximation functional equations involving products of $L(1/2+\alpha_j, \chi_{(1+i)^5d})$ (see Section \ref{sect: apprfcneqn}). In our process, we need to evaluate $\widetilde{F_{n,j}}\left(t\right)$ for $1 \leq j \leq 3$ for a primary $n$,  where
\begin{align}
\label{Fn}
  F_{n,j}(t)=\Phi(t)V_{(\alpha_j)}\left(  \frac {N(n)}{(X t)^{j/2}} \right ).
\end{align}

 To do so, we first note that for any real number $c_u$, we have
\begin{align*}
    \Phi \left(t \right)=\frac 1{2\pi i}\int\limits\limits_{(c_u)}\widehat{\Phi}(u)t^{-u}du.
\end{align*}

   Applying this together with \eqref{eq:Vdef}, we see that for $t>0$,
\begin{align*}
    \widetilde{F_{n,j}}\left(t\right)=
     & 2\int\limits^{\pi/2}_0\int\limits^{\infty}_0\cos (2\pi t r^{1/2}\sin \theta)\frac 1{(2\pi i)^2}\int\limits\limits_{(c_u)}
\int\limits\limits_{(c_s)}\widehat{\Phi}(1+u)r^{-u}
g_{(\alpha_j)}(s) \left(  \frac {N(n)}{ (X r)^{j/2}}  \right )^{-s} du \frac {G_j(s)ds}{s}\frac {\dif r}{r} \dif \theta.
\end{align*}

   We reverse the order of the three inner integrations above to arrive, after some changes of variables (first $r^{1/2} \to r$,
then $2 \pi tr \sin \theta \to r$), at
\begin{align*}
\begin{split}
  \widetilde{F_{n,j}}\left(t\right) =& 4\int\limits^{\pi/2}_0\frac 1{(2\pi i)^2}
\int\limits\limits_{(c_s)}\int\limits\limits_{(c_u)}\widehat{\Phi}(1+u)
g_{(\alpha_j)}(s) \left(  \frac {N(n)}{X^{j/2}}  \right )^{-s}
\int\limits^{\infty}_0\cos (r)\left(\frac r{2\pi t \sin \theta}\right )^{js-2u}du \frac {ds}{s} \frac {\dif r}{r} \dif \theta   \\
     =&
 \frac {4}{(2\pi i)^2}
\int\limits\limits_{(c_s)}\int\limits\limits_{(c_u)}\widehat{\Phi}(1+u)
g_{(\alpha_j)}(s)  \left(  \frac {N(n)}{X^{j/2}}  \right )^{-s}(2\pi t)^{-(js-2u)}\int\limits^{\pi/2}_0 (\sin \theta )^{-(js-2u)} \dif \theta
\int\limits^{\infty}_0\cos (r)r^{js-2u}\frac {\dif r}{r} du \frac {G_j(s)ds}{s}.
\end{split}
\end{align*}

     We note that for $\Re(s)<1$, we have (see \cite[Formula 2, Section 8.380]{GR})
\begin{align}
\label{sinint}
    \int\limits^{\pi/2}_0 (\sin \theta )^{-s} \dif \theta=\frac 12 B(\frac {1-s}{2}, \frac 12)=\frac {\sqrt{\pi}}{2}\frac {\Gamma(\frac {1-s}{2})}{\Gamma(\frac {2-s}{2})},
\end{align}
   where $B(x,y)$ is the Beta function such that when $\Re(x), \Re(y)>0$ (see \cite[Formula 2, Section 8.384]{GR})
\begin{align*}
   B(x, y)=\frac {\Gamma(x)\Gamma(y)}{\Gamma(x+y)},
\end{align*}
   and that (see \cite[Formula 2, Section 8.338]{GR}) $\Gamma(\frac {1}{2})=\sqrt{\pi}$.

  We also note that (see \cite[Formula 9, Section 3.761]{GR}) for any $0<\Re(s)<1$:
\begin{align}
\label{cosMellin}
     \int\limits^{\infty}_0\cos (r)r^{s}\frac {\dif r}{r} =\Gamma (s) \cos \left( \frac {\pi s}{2} \right).
\end{align}

   We now combine \eqref{sinint}, \eqref{cosMellin} and the following relation (see chapter 10 of \cite{Da})
\begin{equation*}
\pi^{-\thalf} 2^{1-u} \cos(\tfrac{\pi}{2} s) \Gamma(s) = \frac{\Gamma \leg{s}{2}}{\Gamma\leg{1-s}{2}}
\end{equation*}
  to see that
\begin{align*}
\int\limits^{\pi/2}_0 (\sin \theta )^{-u} \dif \theta
\int\limits^{\infty}_0\cos (r)r^{u}\frac {\dif r}{r}=\frac {\pi}{2}2^{u-1}\frac{\Gamma \leg{u}{2}}{\Gamma\leg{2-u}{2}}.
\end{align*}
  This implies that
\begin{align*}
\int\limits^{\pi/2}_0 (\sin \theta )^{-(js-2u)} \dif \theta
\int\limits^{\infty}_0\cos (r)r^{js-2u}\frac {\dif r}{r}=\frac {\pi}{2}2^{s-2u-1}\frac{\Gamma \leg{js-2u}{2}}{\Gamma\leg{2-js+2u}{2}}.
\end{align*}

   We then conclude that
\begin{align}
\label{Fnexprssion}
\begin{split}
  \widetilde{F_{n,j}}\left(t\right) =&
 \frac {\pi}{(2\pi i)^2}
\int\limits\limits_{(c_s)}\int\limits\limits_{(c_u)}\widehat{\Phi}(1+u)
g_{(\alpha_j)}(s)  \left(  \frac {N(n)}{X^{j/2}}  \right )^{-s}(\pi t)^{-(js-2u)}\frac{\Gamma \leg{js-2u}{2}}{\Gamma\leg{2-js+2u}{2}} du \frac {G_j(s)ds}{s}.
\end{split}
\end{align}

   Lastly, for any primary $l$ and $n$, we evaluate $\widetilde{F_{n^2l,j }}\left(0\right)$ by applying \eqref{Wt} directly to see that
\begin{align}
\label{F0}
\begin{split}
  \widetilde{F_{n^2l,j}}\left(0\right) =& \int\limits^{\infty}_{-\infty}\int\limits^{\infty}_{-\infty}\Phi \left(N(x+yi) \right)
V_{(\alpha_j)}  \left(  \frac {N(ln^2)}{(X N(x+yi))^{j/2}} \right ) \dif x \dif y \\
=&  \frac {1}{2\pi
   i} \int\limits\limits_{(2)}g_{(\alpha_j)}(s)\left(  \frac {X^{j/2} } {N(ln^2)}\right )^{s}
 \left (\int\limits^{\infty}_{-\infty}\int\limits^{\infty}_{-\infty}\Phi \left(N(x+yi) \right) N(x+yi)^{js/2}
\dif x \dif y \right )  \frac { G_j(s) ds}{s} \\
=& \frac {\pi }{2\pi
   i} \int\limits\limits_{(2)} g_{(\alpha_j)}(s)\left(  \frac {X^{j/2} } {N(ln^2)}\right )^{s}
 \widehat{\Phi}(1+\frac {js}2) \frac {G_j(s)ds}{s},
\end{split}
\end{align}
  since we have
\begin{align*}
  \int\limits^{\infty}_{-\infty}\int\limits^{\infty}_{-\infty}\Phi \left(N(x+yi) \right) N(x+yi)^{js/2}
\dif x \dif y =\int^{2\pi}_0\int^{\infty}_0\Phi (r^2)r^{js}rdrd\theta =\pi \widehat{\Phi}(1+\frac {js}2).
\end{align*}

\subsection{The approximate functional equation}
\label{sect: apprfcneqn}
Let $\chi$ be a primitive quadratic Hecke character $\pmod {m}$ of trivial infinite type defined on $\mathcal{O}_K$.
As shown by E. Hecke, $L(s, \chi)$ admits
analytic continuation to an entire function and satisfies the
functional equation (\cite[Theorem 3.8]{iwakow})
\begin{align}
\label{fneqn}
  \Lambda(s, \chi) = W(\chi)(N(m))^{-1/2}\Lambda(1-s, \overline{\chi}),
\end{align}
   where $|W(\chi)|=(N(m))^{1/2}$ and
\begin{align*}
  \Lambda(s, \chi) = (|D_K|N(m))^{s/2}(2\pi)^{-s}\Gamma(s)L(s, \chi).
\end{align*}

   As $\chi$ is quadratic, we have $\chi=\overline{\chi}$ so that by setting $s=1/2$ in \eqref{fneqn}, we deduce that
\begin{align*}
  W(\chi)=N(m)^{1/2}.
\end{align*}

   Thus, the functional equation in this case becomes
\begin{align}
\label{fneqnquad}
  \Lambda(s, \chi) = \Lambda(1-s, \chi).
\end{align}

  Let $s_j, 1 \leq j \leq n$ be complex numbers for some positive integer $n$. Write ${\bf s}=(s_1, \cdots, s_n)$ and $1-{\bf s}=(1-s_1, \cdots, 1-s_n)$.  Let $G_n(s)$ be an entire, even function, bounded in any strip $-A \leq \Re(s) \leq A$ for some $A>2$ such that $G_n(0)=1$. For some $c >1$, consider the integral
\begin{align*}
   I({\bf s}, \chi)=\frac 1{2 \pi i}\int\limits_{(c)}\prod^n_{j=1}\Lambda(s_j+u, \chi)G_n(u) \frac {\dif u}{u}.
\end{align*}
    Moving the contour of integral to $\Re(u)=-c$, we see that
\begin{equation*}
   \prod^n_{j=1}\Lambda(s_j, \chi)=I({\bf s}, \chi)-
\frac 1{2 \pi i}\int\limits_{(-c)}\prod^n_{j=1}\Lambda(s_j+u, \chi)G_n(u) \frac {\dif u}{u}.
\end{equation*}
   We now apply the functional equation \eqref{fneqnquad} to obtain
\begin{align}
\label{Lambda}
\begin{split}
    \prod^n_{j=1}\Lambda(s_j, \chi)= & I({\bf s}, \chi)
-\frac 1{2 \pi i}\int\limits_{(-c)}\prod^n_{j=1}\Lambda(1-s_j+u, \chi)G_n(u) \frac {\dif u}{u} \\
=& I({\bf s}, \chi)+ \frac 1{2 \pi i}\int\limits_{(-c)}\prod^n_{j=1}\Lambda(1-s_j+u, \chi)G_n(u) \frac {\dif u}{u} \\
=& I({\bf s}, \chi)+I(1-{\bf s}, \chi),
\end{split}
\end{align}
  where the second equality follows from a change of variable $u \rightarrow -u$ in the first integral above.

   Upon expanding $\Lambda(s_i+u), 1 \leq i \leq n$ into convergent Dirichlet series, we have
\begin{align*}
 I({\bf s}, \chi)=&\frac 1{2 \pi i}\int\limits_{(c)} \left ( \sum_{0 \neq \mathcal{A}_1, \cdots, \mathcal{A}_n \subset \mathcal{O}_K} \prod^n_{j=1} \frac{\chi(\mathcal{A}_j)}{N(\mathcal{A}_j)^{s_j+u}}\frac{(|D_K|N(m))^{(u+s_j)/2}\Gamma(s_j+u)}{(2\pi)^{u+s_j}}
 \right )
 G_n(u) \frac {\dif u}{u}, \\
I(1-{\bf s}, \chi)=&
\frac 1{2 \pi i}\int\limits_{(c)} \left ( \sum_{0 \neq \mathcal{A}_1, \cdots, \mathcal{A}_n \subset \mathcal{O}_K} \prod^n_{j=1} \frac{\chi(\mathcal{A}_j)}{N(\mathcal{A}_j)^{1-s_j+u}}\frac{(|D_K|N(m))^{(1-s_j+u)/2}\Gamma(1-s_j+u)}{(2\pi)^{(1-s_j+u)}}
 \right )
 G_n(u) \frac {\dif u}{u}.
\end{align*}

   Applying these expressions and dividing through $\prod^n_{j=1}(|D_K|N(m))^{s_j/2}(2\pi)^{-s_j}\Gamma(s_j)$ on both sides of \eqref{Lambda},
we obtain
\begin{align*}
 \prod^n_{j=1}L(s_j,\chi)=& \frac 1{2 \pi i}\int\limits_{(c)} \left ( \sum_{0 \neq \mathcal{A}_1, \cdots, \mathcal{A}_n \subset \mathcal{O}_K} \prod^n_{j=1} \frac{\chi(\mathcal{A}_j)}{N(\mathcal{A}_j)^{s_j+u}}\frac{(|D_K|N(m))^{u/2}\Gamma(s_j+u)}{(2\pi)^{u}\Gamma(s_j)}
 \right )
 G_n(u) \frac {\dif u}{u} \\
 & + \frac 1{2 \pi i}\int\limits_{(c)} \left ( \sum_{0 \neq \mathcal{A}_1, \cdots, \mathcal{A}_n \subset \mathcal{O}_K} \prod^n_{j=1} \frac{\chi(\mathcal{A}_j)}{N(\mathcal{A}_j)^{1-s_j+u}}\frac{(|D_K|N(m))^{(1-2s_j+u)/2}\Gamma(1-s_j+u)}{(2\pi)^{(1-2s_j+u)}\Gamma(s_j)}
 \right )
 G_n(u) \frac {\dif u}{u}.
\end{align*}

   Recalling that $D_K=-4$, we then deduce from the above by setting $s_j=\frac 12+\alpha_j$, $\chi=\chi_{(1+i)^5d}$ for $d$ odd and square-free that
\begin{align*}
 & \prod^n_{j=1}L(\half+\alpha_j,\chi) \\
=& \sum_{0 \neq \mathcal{A} \subset \mathcal{O}_K} \frac{\chi(\mathcal{A})
\sigma_{(\alpha_n)}(\mathcal{A})}{N(\mathcal{A})^{1/2}}V_{(\alpha_n)} \left(  \frac {N(\mathcal{A})}{N(d)^{n/2}} \right)+N(d)^{- \sum^n_{j=1}\alpha_j}\prod^n_{j=1}\Gamma_{\alpha_j}\sum_{0 \neq \mathcal{A} \subset \mathcal{O}_K} \frac{\chi(\mathcal{A})
\sigma_{-(\alpha_n)}(\mathcal{A})}{N(\mathcal{A})^{1/2}}V_{-(\alpha_n)} \left(  \frac {N(\mathcal{A})}{N(d)^{n/2}} \right).
\end{align*}
   where $\Gamma_{\alpha}$ is defined in \eqref{gamma}, $V_{(\alpha_n)}$ is defined in \eqref{eq:Vdef} and
\begin{align*}
 \sigma_{(\alpha_n)}(\mathcal{A})= \sum_{ \prod^n_{j=1}\mathcal{A}_j=\mathcal{A} }N(\mathcal{A}_j)^{-\alpha_j}.
\end{align*}

  As $\chi_{(1+i)^5d}(\mathcal{A}) \neq 0$ only when $(\mathcal{A}, 2)=1$, in which case we may replace $\mathcal{A}$ by its primary generator. We thus deduce from the above discussions the following approximate functional equation for products of quadratic Hecke $L$-functions.
\begin{lemma}[Approximate functional equation]
\label{lem:AFE}
Let $G_j(s), 1 \leq j \leq 3$ be entire, even functions with rapid decay in the strip $|\Re(s)| \leq 10$ such that $G_j(0)=1, 1 \leq j \leq 3$.  For $\chi_{(1+i)^5d}$ as above, we have
\begin{align}
\label{fcneqnL}
\begin{split}
 \prod^j_{i=1}L(\thalf + \alpha_j, \chi_{(1+i)^5d}) = & \sum_{\substack{n \equiv 1 \bmod {(1+i)^3}}} \frac{\chi_{(1+i)^5d}(n)\sigma_{\alpha_1, \cdots, \alpha_j}(n)}{N(n)^{\frac{1}{2}+\alpha}} V_{(\alpha_j)}
\left(\frac{ N(n)}{N(d)^{j/2}} \right) \\
& + \displaystyle N(d)^{-\sum^j_{i=1}\alpha_i}\Gamma_{\alpha_1, \cdots, \alpha_j}  \sum_{\substack{n \equiv 1 \bmod {(1+i)^3}}} \frac{\chi_{(1+i)^5d}(n)\sigma_{-\alpha_1, \cdots, -\alpha_j}(n)}{N(n)^{\frac{1}{2}-\alpha}}
V_{-(\alpha_j)} \left(\frac{ N(n)}{N(d)^{j/2}} \right),
\end{split}
\end{align}
where $\Gamma_{\alpha_1, \cdots, \alpha_j}= \displaystyle \prod^j_{i=1}\Gamma_{\alpha_i}(s)$ and $\sigma_{\alpha_1, \cdots, \alpha_j}(n)$ is defined in
\eqref{eq:sigma}.
\end{lemma}

\subsection{Analytical behaviors of certain Dirichlet series}
\label{sect: alybehv}

   In this section, we discuss the analytical behaviors of certain Dirichlet series that are needed in our proofs. The first result concerns the analytical behaviors of $A_{\alpha, \beta}(l)$ and $A_{\alpha, \beta, \gamma}(l)$ given in Theorem \ref{theo:recursive2}.
\begin{lemma}
 \label{lemma:A}
 Let $l = l_1 l_2$ and let $A_{\alpha, \beta}(l)$, $A_{\alpha, \beta, \gamma}(l)$ be as in Theorem \ref{theo:recursive2}.  Then both $A_{\alpha, \beta}(l)$ and $A_{\alpha, \beta, \gamma}(l)$ have meromorphic
continuations to $\Re(\alpha), \Re(\beta), \Re(\gamma) > -\half$.  In fact, for any positive integer $M \geq 2$ there exist integers $d_{a,b}, d_{a,b,c}$ (possibly negative or zero) such that
\begin{align}
\label{eq:Aproduct}
\begin{split}
  A_{\alpha,\beta}(l) =& C_{\alpha,\beta}(l)\prod_{\substack{1 \leq a + b\leq M-1 \\ a, b, c \geq 0}} \zeta_K(a+b+ 2a\alpha + 2b\beta)^{d_{a,b}}, \\
 A_{\alpha,\beta,\gamma}(l) =& C_{\alpha,\beta,\gamma}(l)\prod_{\substack{1 \leq a + b + c \leq M-1 \\ a, b, c \geq 0}} \zeta_K(a+b+c + 2a\alpha + 2b\beta + 2c\gamma)^{d_{a,b,c}},
\end{split}
\end{align}
where for any $\delta > 0$, $C_{\alpha,\beta}(l), C_{\alpha,\beta,\gamma}(l)$ are given by absolutely convergent Euler products in the region
$\Re(\alpha), \Re(\beta), \Re(\gamma) > \frac{1-M}{2M} + \delta$.  Moreover, in this region $C_{\alpha,\beta}(l), C_{\alpha, \beta, \gamma}(l)$ satisfy the bound
\begin{equation*}
 C_{\alpha,\beta}(l), \quad C_{\alpha,\beta,\gamma}(l) \ll \sqrt{N(l_1)} N(l)^{\varepsilon}.
\end{equation*}
\end{lemma}

   The proof of the above lemma is similar to that of \cite[Lemma 4.1]{Young1} and \cite[Lemma 4.1]{Sono}, so we shall omit it here.  We only note here that when $a+b=1$ or $a + b + c =1$ then we have $d_{a,b}=d_{a,b,c} = 1$ and this readily implies the analytical behaviors of $B_{\alpha,\beta}(l)$ and
 $B_{\alpha,\beta,\gamma}(l)$ defined in \eqref{B}.

   To facilitate our treatments in the proof of Theorem \ref{theo:recursive2}, we shall make the following remark similar to \cite[Remark 2.2]{Young1} and \cite[Remark 2.2]{Young2}.
\begin{remark}
\label{remark:zero}
 We choose $G_1$ so that $G_1(\pm \alpha) = G_1(\thalf \pm \alpha) = 0$. We choose $G_j(s), j=2,3$
to vanish at the poles of all the $\zeta_K$'s which occur in \eqref{eq:Aproduct} as numerators (i.e., with $d_{a,b}>0$ or $d_{a,b,c} > 0$) in the corresponding factorization of  $A_{\alpha + s, \beta+s}$ or $A_{\alpha + s, \beta+s, \gamma+s}$, and also
to be divisible by all the $\zeta_K$'s which occur in \eqref{eq:Aproduct} as denominators (i.e., with $d_{a,b}<0$ or $d_{a,b,c} < 0$) in the corresponding factorization of  $A_{\alpha + s, \beta+s}$ or $A_{\alpha + s, \beta+s, \gamma+s}$ for $M$ large enough so that $A_{\alpha + s, \beta+s}, A_{\alpha + s, \beta+s, \gamma+s}$ have meromorphic continuations to $\Re(s) > -\half + \varepsilon$ for a given $\epsilon>0$. We also assume that $G_j(s)$ is symmetric under any permutation of $\{\alpha, \beta, \gamma\}$, and under switching any $\alpha, \beta, \gamma$ with its negative, and under switching $s$ with $-s$.
\end{remark}

Let $g(k,n)$ be defined as in \eqref{g2}. We now fix a generator for every prime ideal $(\varpi) \in \mathcal{O}_K$ together with $1$ as the generator for
the ring $\mz[i]$ itself and
extend to any ideal of $\mathcal{O}_K$ multiplicatively. We denote the set of such generators by $G$. Let $k_1 \in \mathcal{O}_K$ be square-free and
$(l, a)=1$ for a primary element $l \in \mathcal{O}_K$. For fixed integer $j \geq 1$ and complex numbers $\alpha_i, 1 \leq i \leq j$, we define
$J_{k_1,j}(v,w;l, a)$ as
\begin{align}
\label{eq:J}
  J_{k_1,j}(v,w;l, a) &=\sum_{\substack{n \equiv 1 \bmod {(1+i)^3} \\ (n,a)=1}}  \sum_{\substack {k_2 \neq 0 \\ k_2 \in \mathcal{O}_K}}
\frac {\sigma_{\alpha_1, \cdots, \alpha_j}(n)}{N(n)^{w}N(k_2)^{v}} \frac {g(k_1k^2_2,ln)}{N(ln)},
\end{align}
    where we use the convention throughout the paper that all sums over $k_2$ are restricted to $k_2 \in G$.

  Our next lemma gives the analytic properties of $J_{k_1,j}(v,w;l, a) $.
\begin{lemma}
\label{lemma:Jprop}
 Suppose that $l$ is primary such that $(l, 2a) =1$, $k_1$ is square-free, and $J_{k_1,j}(v,w;l, a) $ is given by \eqref{eq:J} for $\Re(v) > 2$ and $\Re(w) > 2$.  Then $J_{k_1,j}(v,w;l, a)$ has a meromorphic continuation to $\Re(v) \geq 2$ and $\Re(w) > \delta$ for any $\delta > 0$, provided that $\alpha_i, 1 \leq i \leq j$ are small enough compared to $\delta$.   Moreover, in this region we have
\begin{equation*}
  J_{k_1,j}(v,w;l, a) =\prod^j_{i=1} L_{2al}(\thalf + w + \alpha_i, \chi_{ik_1})  I_{k_1,j}(v, w),
\end{equation*}
 where $I_{k_1,j }(v, w)$ is analytic in this region and satisfies the bound
\begin{equation*}
 I_{k_1,j}(v, w) \ll_{\delta, \varepsilon} N(l)^{-\half + \varepsilon} .
\end{equation*}
\end{lemma}

  The proof of the above lemma is similar to that of \cite[Lemma 5.1]{Young1}, \cite[Lemma 5.2]{Young2} and \cite[Lemma 4.3]{Sono}, so we omit it here.

\section{Proof of Theorem \ref{theo:recursive2}}
\label{sec: pfrecursive}

\subsection{Initial Treatment}

  We fix $\alpha_1=\alpha, \alpha_2=\beta, \alpha_3=\gamma$ throughout and we identify $M_{(\alpha_j)}(l)$ with $M_{\alpha}(l), M_{\alpha, \beta}(l)$ and  $M_{\alpha, \beta, \gamma}(l)$ with $j=1,2,3$, respectively. This applies to similar notations such as $g_{(\alpha_j)}, \sigma_{(\alpha_j)}, V_{(\alpha_j)}$  as well. We apply the approximate functional equation \eqref{fcneqnL} for a fixed  $1 \leq j \leq 3$ to write $M_{(\alpha_j)}(l) = M_1((\alpha_j),l) + M_{-1}((\alpha_j),l)$, where
\begin{align*}
 M_1((\alpha_j),l) =&  \sumstar_{(d,2) = 1}F(N(d)) \sum_{\substack{ n \equiv 1 \bmod {(1+i)^3}}}
\frac{\chi_{(1+i)^5d}(nl) \sigma_{(\alpha_j)}(n)}{N(n)^{\frac{1}{2}}}V_{(\alpha_j)}  \left(\frac{ N(n)}{N(d)^{j/2}}\right), \\
 M_{-1}((\alpha_j),l) =& \Gamma_{(\alpha_j)} \sumstar_{(d,2) = 1} N(d)^{-\sum^j_{i=1}\alpha_i} F(N(d))
\sum_{\substack{ n \equiv 1 \bmod {(1+i)^3}}} \frac{\chi_{(1+i)^5d}(nl)  \sigma_{-(\alpha_j)}(n)}{N(n)^{\thalf}} V_{-(\alpha_j)}
\left(\frac{N(n)}{N(d)^{j/2}}\right).
\end{align*}

We shall make the convention that we may often drop the dependence on $(\alpha_j)$ and $l$ to simply write $M_1$, $M_2$ and other expressions when there is no risk of confusion. We shall also mainly focus on evaluating $M_1$ as the evaluation of $M_2$ can be done by noticing the following remark:
\begin{remark}
\label{remark:M1toM2}
To derive an expression for $M_2$ via a corresponding term from $M_1$ involves swapping $\alpha_i$ and $-\alpha_i$, $1 \leq i \leq 3$, replacing $F(x)$ by $F_{-(\alpha_j)}(x) = x^{-\sum^j_{i=1}\alpha_i} F(x)$, and multiplying by $\Gamma_{(\alpha_j)}$, in that order.
\end{remark}

  We now apply the M\"{o}bius inversion to remove the square-free condition over $d$ in $M_1$ and $M_2$. Let $\mu_{[i]}$ be the M\"{o}bius function in $\mathcal{O}_{K}$, we have
\begin{equation*}
 M_1 = \sum_{\substack{ a \equiv 1 \bmod {(1+i)^3} \\ (a,2l) = 1}}  \mu_{[i]}(a) \sum_{(d,2)=1} F(N(d a^2))
\sum_{\substack{ n \equiv 1 \bmod {(1+i)^3}\\ (n,2a)=1}}
\frac{\chi_{(1+i)^5d}(nl)\sigma_{(\alpha_j)}(n)}{N(n)^{\frac{1}{2}}} V_{(\alpha_j)} \left(\frac{ N(n)}{N(a^2d)^{j/2}}\right).
\end{equation*}
Now we separate the terms with $N(a) \leq Y$ and with $N(a) > Y$ ($Y$ a parameter to be chosen later), writing $M_1 = M_N + M_R$, respectively.  We similarly write $M_{-1} = M_{-N} + M_{-R}$.

\subsection{Estimating $M_R$: applying the recursion}
\label{section:MR}

   We now make a change of variable by letting $d \rightarrow b^2 d$ with the new $d$ being square-free to see that
\begin{equation*}
M_R = \sum_{\substack{ a \equiv 1 \bmod {(1+i)^3} \\ (a,2l) = 1\\ N(a) > Y}}\mu_{[i]}(a)
  \sum_{\substack{ b \equiv 1 \bmod {(1+i)^3} \\ (b,2l)=1}}  \sumstar_{(d,2)=1}   F(N(d (ab)^2))
\sum_{\substack{ n \equiv 1 \bmod {(1+i)^3} \\ (n,2ab)=1}}  \frac{\chi_{(1+i)^5d}(nl)\sigma_{(\alpha_j)}(n)}{N(n)^{\frac{1}{2}}}
V_{(\alpha_j)}  \left(\frac{ N(n)}{N((ab)^2d)^{j/2}}\right).
\end{equation*}

   We further let $c=ab$ to obtain
\begin{align*}
M_R = \sum_{\substack{ c \equiv 1 \bmod {(1+i)^3} \\ (c,2l)=1}}  \sum_{\substack{ a \equiv 1 \bmod {(1+i)^3} \\ a |c \\ N(a) > Y}} \mu_{[i]}(a)  \sumstar_{(d,2)=1}  F\left(N(d c^2) \right)
\sum_{\substack{ n \equiv 1 \bmod {(1+i)^3} \\ (n,2c)=1}} \frac{\chi_{(1+i)^5d}(nl)\sigma_{(\alpha_j)}(n)}{N(n)^{\frac{1}{2}}}
V_{(\alpha_j)} \left(\frac{ N(n)}{N(c^2d)^{j/2}}\right).
\end{align*}
Using the definition of $V_{(\alpha_j)} $ as an integral representation given in \eqref{eq:Vdef}, we see that the inner sum over $n$ above is
\begin{equation*}
 \sum_{\substack{ n \equiv 1 \bmod {(1+i)^3} \\ (n,2ab)=1}}  \frac{\chi_{(1+i)^5d}(nl)\sigma_{(\alpha_j)}(n)}{N(n)^{\frac{1}{2}}}
 \frac{1}{2\pi i} \int\limits\limits_{(2)} \frac{G_j(s)}{s} g_{(\alpha_j)}(s) \frac{N(c^2 d)^{js/2}}{N(n)^s} ds.
\end{equation*}

  We move the sum over $n$ inside the integral to get
\begin{align*}
M_R =& \sum_{\substack{ c \equiv 1 \bmod {(1+i)^3} \\ (c,2l)=1}} \sum_{\substack{ a \equiv 1 \bmod {(1+i)^3} \\ a |c \\ N(a) > Y}}  \mu_{[i]}(a) \sumstar_{(d,2)=1} \chi_{(1+i)^5d}(l) F\left( N(d c^2) \right)
\\
& \times \frac{1}{2 \pi i} \int\limits\limits_{(\thalf + \varepsilon)} (N(c^2d))^{js/2} \prod^j_{i=1}L(\thalf + \alpha_i + s, \chi_{(1+i)^5d})
\prod^j_{i=1} \prod_{\substack{ \varpi_j \equiv 1 \bmod {(1+i)^3} \\ \varpi_j | c}}\Big(1-\frac {\chi_{(1+i)^5d}(\varpi_j)}{N(\varpi_j)^{(1/2+\alpha+s)}} \Big ) \frac{G_j(s)}{s} g_{(\alpha_j)}(s) ds.
\end{align*}
   We now move the line of integration to $\varepsilon$ without crossing any poles in this process by Remark \ref{remark:zero}.  Then expanding $\displaystyle \prod_{\substack{ \varpi_j \equiv 1 \bmod {(1+i)^3} \\ \varpi_j | c}}\Big(1-\frac {\chi_{(1+i)^5d}(\varpi_j)}{N(\varpi_j)^{(1/2+\alpha+s)}} \Big )$, we obtain that
\begin{align}
\label{eq:prerecursion}
\begin{split}
M_R =& \sum_{\substack{ c \equiv 1 \bmod {(1+i)^3} \\ (c,2l)=1}}  \sum_{\substack{ a \equiv 1 \bmod {(1+i)^3} \\ a |c \\ N(a) > Y}} \mu_{[i]}(a)
\sum_{\substack{ r_i \equiv 1 \bmod {(1+i)^3} \\ 1 \leq i \leq 3 \\ r_i | c}}
\prod^j_{i=1}\frac{\mu_{[i]}(r_i)}{N(r_i)^{\half + \alpha_i}}
\\
& \times \frac{1}{2 \pi i} \int\limits\limits_{(\varepsilon)} \sumstar_{(d,2)=1} \chi_{(1+i)^5d}(l\prod^j_{i=1}r_i) F_{\frac{js}{2};N(c^2)}(N(d))\Big(N(r_i)\Big )^{-s} \prod^j_{i=1}L(\thalf + \alpha_i + s, \chi_{(1+i)^5d})\frac{G_j(s)}{s} g_{(\alpha_j)}(s) ds,
\end{split}
\end{align}
where $F_{\nu;y}(x) = (xy)^{\nu} F(xy)$ and $\varepsilon \asymp (\log{X})^{-1}$.

 Note that the inner sum over $d$ above is of the form $M_{(\alpha_j+s)}(l\prod^j_{i=1}r_i)$, but with a new weight function with smaller support ($N(d) \asymp X/N(c)^2$). Now we truncate the integral in \eqref{eq:prerecursion} so that $|\Im(s)| \leq (\log(X/N(c^2))^2$.  When $N(c)^2 \leq X^{1-\varepsilon}$, the exponential decay of the integrand implies that the error introduced by this truncation is negligible.  While when $N(c)^2 \geq X^{1-\varepsilon}$, the sum over $d$ is almost bounded so that the convexity bound $L(1/2 + \alpha + s, \chi_{(1+i)^5d}) \ll ((1+|s|)^2N(d))^{1/4+\varepsilon}$
implies that the error introduced is of size $O(X^{\half + \varepsilon})$. We can then apply Theorem \ref{theo:recursive2} to the truncated integral,
and the same argument as above allows us to extend the integral back to the whole vertical line, without introducing a new error.
In this way, we can express $M_R$ as the sum of $2^j$ main terms plus an error of size
\begin{align*}
 \ll & X^{\varepsilon} \sum_{\substack{ c \equiv 1 \bmod {(1+i)^3} \\ (c,2l)=1}}   \sum_{\substack{ a \equiv 1 \bmod {(1+i)^3} \\ a |c \\ N(a) > Y}}|\mu_{[i]}(a)| \sum_{\substack{ r_i \equiv 1 \bmod {(1+i)^3} \\ 1 \leq i \leq 3 \\ r_i | c}}
\prod^j_{i=1}\frac{|\mu_{[i]}(r_i)|}{N(r_i)^{\half}}
N(l \prod^j_{i=1}r_i) ^{1/2 + \varepsilon} \leg{X}{N(c)^2}^{f + \varepsilon} \\
\ll & \frac{X^{f + \varepsilon}}{Y^{2f - 1}} N(l)^{1/2 + \varepsilon}.
\end{align*}

For the main terms, by a direct application of  Theorem \ref{theo:recursive2}, we see  that
\begin{align*}
 & M_R(\epsilon_1, \cdots, \epsilon_j) \\
=& \pi \sum_{\substack{ c \equiv 1 \bmod {(1+i)^3} \\ (c,2l)=1}}  \sum_{\substack{ a \equiv 1 \bmod {(1+i)^3} \\ a |c \\ N(a) > Y}} \mu_{[i]}(a)
\sum_{\substack{ r_i \equiv 1 \bmod {(1+i)^3} \\ 1 \leq i \leq 3 \\ r_i | c}}
\prod^j_{i=1}\frac{\mu_{[i]}(r_i)}{N(r_i)^{\half + \alpha_i}}
\frac{1}{\sqrt{N((l\prod^j_{i=1}r_i)^*)}} \frac{1}{2 \zeta_{K,2}(2)}
\\
& \times \frac{1}{2\pi i} \int\limits\limits_{(\varepsilon)} A_{\epsilon_1(\alpha_1+ s), \cdots, \epsilon_j(\alpha_j+ s)}(l\prod^j_{i=1}r_i)
\Gamma_{\alpha_1+ s, \cdots, \alpha_j + s}^{\delta_1, \cdots, \delta_j}\widehat{F}_{js/2;N(c^2)}(w) \frac{1}{N(\prod^j_{i=1}r_i)^{s}} \frac{G_j(s)}{s} g_{(\alpha_j)}(s) ds,
\end{align*}
  where we denote $w = 1 - \delta_1(\alpha_1+s) - \cdots- \delta_j(\alpha_j+s)$.

   Now we apply the relation
\begin{equation*}
 \widehat{F}_{js/2;N(c^2)}(u) = \int_0^{\infty} (xN(c^2))^{\frac{js}{2}} F(N(c^2)x) x^{u} \frac{dx}{x} = N(c)^{-2u} \widehat{F}(\tfrac{js}{2} + u)
\end{equation*}
  to see that
\begin{align}
\label{eq:MRmainterms}
\begin{split}
M_R(\epsilon_1, \cdots, \epsilon_j) =&  \frac{\pi }{2 \zeta_{K,2}(2)}\sum_{\substack{ c \equiv 1 \bmod {(1+i)^3} \\ (c,2l)=1}}\frac{1}{N(c)^{2w}}  \sum_{\substack{ a \equiv 1 \bmod {(1+i)^3} \\ a |c \\ N(a) > Y}}\mu_{[i]}(a)
\frac{1}{2\pi i} \int\limits\limits_{(\varepsilon)} \Gamma_{\alpha_1+ s, \cdots, \alpha_j + s}^{\delta_1, \cdots, \delta_j}
\frac{G_j(s)}{s} g_{(\alpha_j)}(s)
\widehat{F}(\tfrac{js}{2} + w)
\\
& \times \sum_{\substack{ r_i \equiv 1 \bmod {(1+i)^3} \\ 1 \leq i \leq 3 \\ r_i | c}}
\prod^j_{i=1}\frac{\mu_{[i]}(r_i)}{N(r_i)^{\half + \alpha_i+s}}\frac{1}{\sqrt{N((l\prod^j_{i=1}r_i)^*)}}
 A_{\epsilon_1(\alpha_1+ s), \cdots, \epsilon_j(\alpha_j+ s)}(l\prod^j_{i=1}r_i)
  ds.
\end{split}
\end{align}

   We summarize our discussions above in the following result.
\begin{lemma}
\label{lemma:MRresult}
If Theorem \ref{theo:recursive2} holds with a parameter $f > 1/2$ when $j=1,2$ and $f > 3/4$ when $j=3$, then
\begin{align*}
 M_R = \sum_{\epsilon_1, \cdots \epsilon_j \in \{\pm 1\}} M_R(\epsilon_1, \cdots, \epsilon_j) + O(\sqrt{N(l)} \frac{X^{f + \varepsilon}}{Y^{2f-1}})+O(X^{1/2+\varepsilon}),
\end{align*}
where $M_R(\epsilon_1, \cdots, \epsilon_j)$ is defined in \eqref{eq:MRmainterms}.
\end{lemma}

\subsection{Estimating $M_R$: further simplifications}
\label{section:proofofMRbound}

    In this section, we show that some of the main terms appearing in Lemma \ref{lemma:MRresult} can be treated as error terms as well by establishing
\begin{lemma}
\label{lemma:MRbound}
 If at least two of the $\epsilon_i$'s are $-1$, then for a special choice of $G_j(s)$ described in Remark \ref{remark:zero}, we have
\begin{equation}
\label{eq:twonegeps}
 M_R(\epsilon_1, \epsilon_2) \ll  Y X^{1/2} (N(l) X)^{\varepsilon}, \quad M_R(\epsilon_1, \epsilon_2, \epsilon_3) \ll  Y X^{3/4} (N(l) X)^{\varepsilon},
\end{equation}
and furthermore,
\begin{equation}
\label{eq:threenegeps}
 M_R(-1, -1, -1) \ll X^{3/4} (N(l) X)^{\varepsilon}.
\end{equation}
\end{lemma}
\begin{proof}
 Since the proofs are similar, we prove \eqref{eq:twonegeps} only for $M_R(-1, -1, 1)$ here. We extend the sum over $a$ to all primary integers in $K$,
and subtract the contribution from $N(a) \leq Y$, getting $M_R(-1,-1,1) = M'(-1,-1,1) - M''(-1,-1,1)$, respectively.
To treat $M'(-1,-1,1)$, we note that the sum over $a$ becomes
$\displaystyle \sum_{\substack{ a \equiv 1 \bmod {(1+i)^3} \\ a | c}} \mu_{[i]}(a)$, which is not $0$ only when $c=1$.  This implies that $c=r_1 = r_2 =r_3 =1$ so that
\begin{align*}
M'(-1, -1, 1) =  \frac{\pi}{2 \zeta_{K,2}(2)}
\frac{1}{2\pi i} \int\limits_{(\varepsilon)} \Gamma_{\alpha + s} \Gamma_{\beta + s}
\frac{G_3(s)}{s} g_{\alpha,\beta,\gamma}(s)
\widehat{F}(1-\alpha-\beta-\tfrac{s}{2})
\frac{1}{\sqrt{N(l_1)}}
 A_{-\alpha - s, -\beta - s, \gamma + s}(l)
  ds.
\end{align*}

   In view of Lemma \ref{lemma:A} and our choice of $G_3(s)$ described in Remark \ref{remark:zero}, we can move
 the contour of integration to $\Re(s) = \half - \delta$. Using the bound
\begin{equation*}
 \widehat{F}(1-\alpha-\beta- \tfrac{s}{2}) \ll X^{1 - \tfrac{\sigma}{2}},
\end{equation*}
  we see that
\begin{equation*}
 M'(-1,-1,1) \ll N(l_1)^{-\half} N(l_1)^{\half+\varepsilon} N(l)^{\varepsilon} X^{1-1/4 + \varepsilon}.
\end{equation*}

  As for $M''(-1,-1,1)$, we write $c=ab$ and note that the condition $r_1, r_2, r_3| ab$ is equivalent to $[r_1,r_2,r_3]/(a, [r_1,r_2,r_3]) |b$. We can then write $b=kr_1,r_2,r_3]/(a, [r_1,r_2,r_3])$ with $k$ being primary and $(k, 2l)=1$. On summing over $k$ first, we obtain that
\begin{align*}
\begin{split}
& M''(-1, -1, 1) \\
=&   \frac{\pi}{2 \zeta_{K,2}(2)} \sum_{\substack{ a \equiv 1 \bmod {(1+i)^3} \\ (a,2l) = 1\\ N(a) \leq Y}}\frac{\mu_{[i]}(a)}{N(a)^{2-2\alpha-2\beta-4s}}
\frac{1}{2\pi i} \int\limits_{(\varepsilon)} \Gamma_{\alpha + s} \Gamma_{\beta + s}
\frac{G_3(s)}{s} g_{\alpha,\beta,\gamma}(s)
\widehat{F}(1-\alpha-\beta-\tfrac{s}{2})
\\
& \times \sum_{\substack{ r_1, r_2, r_3 \equiv 1 \bmod {(1+i)^3} \\ (r_1 r_2 r_3, 2l)=1}}
 \frac{\mu_{[i]}(r_1) \mu_{[i]}(r_2) \mu_{[i]}(r_3) }{N(r_1)^{\half + \alpha+s} N(r_2)^{\half + \beta+s} N(r_3)^{\half + \gamma+s}}
\frac{1}{\sqrt{N((lr_1 r_2 r_3)^*)}}
\leg{N((a,[r_1, r_2, r_3]))}{N([r_1, r_2, r_3])}^{2-2\alpha-2\beta-4s} \\
& \times \zeta_{K, 2l}(2-2\alpha-2\beta - 4s)  A_{-\alpha - s, -\beta - s, \gamma + s}(lr_1 r_2 r_3)
  ds.
\end{split}
\end{align*}

  Again we move the contour of integration to $\Re(s) = \half - \delta$ and bound everything trivially to see that
\begin{align*}
 & M''(-1,-1,1) \\
\ll & X^{3/4 + \delta/2}\sum_{\substack{ a \equiv 1 \bmod {(1+i)^3} \\ (a,2l) = 1\\ N(a) \leq Y}} N(a)^{4\delta}
\sum_{\substack{ r_1, r_2, r_3 \equiv 1 \bmod {(1+i)^3} \\ (r_1 r_2 r_3, 2l)=1}}
\frac{N(lr_1 r_2 r_3)^{\varepsilon} |\mu_{[i]}(r_1) \mu_{[i]}(r_2) \mu_{[i]}(r_3)|}{N(r_1 r_2 r_3)^{1- \delta}}
 \leg{N((a,[r_1, r_2, r_3]))}{N([r_1, r_2, r_3])}^{4\delta}.
\end{align*}
  We apply the bound $N((a, [r_1, r_2, r_3])) \leq N(a)$ and note that the sums over
$r_1, r_2, r_3$ converge absolutely for any $\delta > 0$ by taking $\varepsilon$ small enough compared to $\delta$ to see that
\begin{equation*}
 M''(-1,-1,1) \ll Y X^{3/4} (N(l) X)^{\varepsilon}.
\end{equation*}

  Lastly, we bound $M_R(-1, -1, -1)$ by writing $c=ab$ again to see that
\begin{align*}
& M_R(-1, -1, -1) \\
 =&  \frac{1}{2 \zeta_{K,2}(2)} \sum_{\substack{ a \equiv 1 \bmod {(1+i)^3} \\ (a,2l) = 1\\ N(a) > Y}}  \mu_{[i]}(a)
\frac{1}{2\pi i} \int\limits\limits_{(\varepsilon)} \Gamma_{\alpha + s, \beta + s, \gamma + s}
\frac{G_3(s)}{s} g_{\alpha,\beta,\gamma}(s)
\widehat{F}(1- \alpha-\beta-\gamma-\tfrac{3s}{2}) \\
& \times \sum_{\substack{ b \equiv 1 \bmod {(1+i)^3} \\ (b,2l)=1}}\frac{1}{N(ab)^{2(1-\alpha-\beta-\gamma-3s)}}
 \sum_{\substack{ r_1, r_2, r_3 \equiv 1 \bmod {(1+i)^3} \\ r_1, r_2, r_3 | ab}}
\frac{\mu_{[i]}(r_1) \mu_{[i]}(r_2) \mu_{[i]}(r_3) }{N(r_1)^{\half + \alpha+s} N(r_2)^{\half + \beta+s} N(r_3)^{\half + \gamma+s}} \\
& \times \frac{A_{-\alpha - s, -\beta - s, -\gamma - s}(lr_1 r_2 r_3)}{\sqrt{N((lr_1 r_2 r_3)^*)}}
  ds.
\end{align*}
 We now move the contour of integration to $\Re(s)  = \frac16 - \delta$. This leads to the desired bound given in \eqref{eq:threenegeps} by noting that
the sums over $a$ and $b$ converge absolutely.
\end{proof}

   Combining Lemma \ref{lemma:MRresult} and Lemma \ref{lemma:MRbound}, we deduce that
\begin{lemma}
\label{lemma:MRresult1}
If Theorem \ref{theo:recursive2} holds with a parameter $f > 1/2$ when $j=1,2$ and $f > 3/4$ when $j=3$, then
\begin{align*}
\begin{split}
 M_R(\alpha, l) =& M_R(1) + M_R(-1) + O(\sqrt{N(l)} \frac{X^{f + \varepsilon}}{Y^{2f-1}})+O(X^{1/2+\varepsilon}), \\
 M_R(\alpha, \beta, l) =& M_R(1,1) + M_R(1,-1) + M_R(-1,1)+O(\sqrt{N(l)} \frac{X^{f + \varepsilon}}{Y^{2f-1}})+O(Y X^{1/2} (N(l) X)^{\varepsilon}), \\
 M_R(\alpha, \beta, \gamma, l) =& M_R(1,1,1) + M_R(1,1,-1) + M_R(1,-1, 1)+M_R(-1,1,1) +O(\sqrt{N(l)} \frac{X^{f + \varepsilon}}{Y^{2f-1}})+O(Y X^{3/4} (N(l) X)^{\varepsilon}).
\end{split}
\end{align*}
\end{lemma}

\subsection{Computing $M_N$: applying Poisson summation}
\label{section:MN}

  We recall that
\begin{equation*}
M_N =  \sum_{\substack{ a \equiv 1 \bmod {(1+i)^3} \\(a,2l)=1 \\ N(a) \leq Y}}  \mu_{[i]}(a)  \sum_{\substack{ n \equiv 1 \bmod {(1+i)^3} \\(n,2a)=1 }}
 \frac{\leg{(1+i)^5}{nl}\sigma_{(\alpha_j)}(n) }{N(n)^{\frac{1}{2}}} \sum_{(d,2)=1} \leg{d}{nl} \Phi\left(\frac{N(d a^2)}{X}\right)
V_{(\alpha_j)} \left(\frac{N(n)}{N(a^2d)^{j/2}}\right).
\end{equation*}

  We now apply the Poisson summation formula given in Lemma \ref{Poissonsumformodd} to see that
\begin{align*}
 & \sum_{(d,2)=1} \leg{d}{nl} \Phi\left(\frac{N(d a^2)}{X}\right)
V_{(\alpha_j)}  \left(\frac{N(n)}{N(a^2d)^{j/2}}\right)
=  \frac {X}{2N(a^2nl)} \leg {1+i}{nl} \sum_{k \in
   \mz[i]}(-1)^{N(k)}g(k, nl)\widetilde{F_{n,j}}\left(\sqrt{\frac {N(k)X}{2N(a^2nl)}}\right),
\end{align*}
   where $F_{n,j}(t)$ is given in \eqref{Fn}. We then deduce that
\begin{align*}
 M_N = \frac{X}{2}  \sum_{\substack{ a \equiv 1 \bmod {(1+i)^3} \\(a,2l)=1 \\ N(a) \leq Y}}  \frac{\mu_{[i]}(a)}{N(a)^2}
\sum_{\substack{ n \equiv 1 \bmod {(1+i)^3} \\ (n,2a)=1}}
\frac{\sigma_{(\alpha_j)}(n)}{N(n)^{\thalf}}  \sum_{k \in
   \mz[i]}(-1)^{N(k)} \frac{g(k, ln)}{N(ln)}
\widetilde{F_{n,j}}\left(\sqrt{\frac {N(k)X}{2N(a^2nl)}}\right).
\end{align*}

Now we write $M_N = M_N(k=0) + M_N(k \neq 0)$, where $M_N(k=0)$ corresponds to the term with $k=0$.

\subsection{Computing $M_N$: the term $M_N(k=0)$}
Note that by Lemma \ref{Gausssum} we have $g(0,n)=\varphi_{[i]}(ln)$ if $ln = \square$
(i.e. $n = l_1 \square$), and $0$ otherwise.   Thus we get
\begin{align*}
  M_N(k=0)=& \frac {X}{2} \sum_{\substack{ a \equiv 1 \bmod {(1+i)^3} \\(a,2l)=1 \\ N(a) \leq Y}} \frac {\mu_{[i]}(a)}{N(a)^2} \sum_{\substack{ n \equiv 1 \bmod {(1+i)^3} \\ (n,2a)=1}}
\frac{\sigma_{(\alpha_j)}(l_1n^2)}{N(l_1n^2)^{\frac 12}}\frac{\varphi_{[i]}(ln)}{N(ln)}\widetilde{F_{n^2l_1}}\left(0\right).
\end{align*}

    We then apply \eqref{F0} to deduce that
\begin{align*}
  M_N(k=0)=& \frac {\pi}{2} \sum_{\substack{ a \equiv 1 \bmod {(1+i)^3} \\(a,2l)=1 \\ N(a) \leq Y}} \frac {\mu_{[i]}(a)}{N(a)^2} \times
\frac {1}{2\pi
   i} \int\limits\limits_{(2)} g_{(\alpha_j)}(s)
 \widehat{F}(1+\frac {js}2)D_N(k=0;s)  \frac {G_j(s)ds}{s},
\end{align*}
  where
\begin{equation*}
 D_N(k=0;s) = \sum_{\substack{ n \equiv 1 \bmod {(1+i)^3} \\ (n,a)=1}}
\frac{\sigma_{(\alpha_j)}(l_1n^2)}{N(l_1n^2)^{\frac 12 +s}}\frac{\phi_{[i]}(ln)}{N(ln)}.
\end{equation*}

\begin{lemma}
\label{lemma:MNk0}
 For special choices of $G_j(s), 1 \leq j \leq 3$ described in Remark \ref{remark:zero}, we have for $j=1$, $l$ primary and square-free,
\begin{align}
\label{eq:right}
M_{N}(k=0) + M_R(1) =  \frac {\pi }{2\zeta_{K,2}(2)\sqrt{N(l)}}
\frac{1}{2\pi i} \int_{(\varepsilon)} \widehat{F}(1 + \tfrac{s}{2})\frac{G_1(s)}{s} g_{\alpha}(s) A_{\alpha+s}(l)  ds.
\end{align}
  For $j=2,3$ and a general primary $l$,
\begin{equation}
\label{eq:MNk0andMR111}
 M_N(k=0) + M_R(1,\cdots, 1) = \pi A_{\alpha_1, \cdots, \alpha_j}(l) \frac{\widehat{F}(1)}{2 \zeta_{K,2}(2) \sqrt{N(l_1)}} + O(X^{(1-\frac j{4}) + \varepsilon} N(l)^{\varepsilon}).
\end{equation}
\end{lemma}
\begin{proof}
  The expression given in \eqref{eq:right} can be established by proceeding similarly to the treatment in Section 6.2 of \cite{Young1}. To prove \eqref{eq:MNk0andMR111}, we use the expression \eqref{eq:MRmainterms} by writing $c=ab$ there to see that
\begin{equation*}
 M_R(1,\cdots,1) =  \frac {\pi}{2} \sum_{\substack{ a \equiv 1 \bmod {(1+i)^3} \\(a,2l)=1 \\ N(a) > Y}} \frac {\mu_{[i]}(a)}{N(a)^2} \times
\frac {1}{2\pi
   i} \int\limits\limits_{(2)} g_{\alpha_1, \cdots, \alpha_j}(s)
 \widehat{F}(1+\frac {3s}2)D_R(1, \cdots, 1;s)  \frac {G_j(s)ds}{s},
\end{equation*}
where
\begin{align*}
 D_R(1,\cdots, 1;s) = &  \frac{1}{\zeta_{K,2}(2)}  \sum_{\substack{ b \equiv 1 \bmod {(1+i)^3} \\ (b,2l)=1}}\frac{1}{N(b)^2}\sum_{\substack{ r_i \equiv 1 \bmod {(1+i)^3} \\ 1 \leq i \leq 3 \\ r_i | c}}
\prod^j_{i=1}\frac{\mu_{[i]}(r_i)}{N(r_i)^{\half + \alpha_i+s}}
\\
& \times \sum_{\substack{ n\equiv 1 \bmod {(1+i)^3} \\ (n,2) = 1}} \frac{\sigma_{\alpha_1, \cdots, \alpha_j}((l\prod^j_{i=1}r_i)^* n^2)}{N((l \prod^j_{i=1}r_i)^* n^2)^{\frac12 + s}} \prod_{\substack{ \varpi \equiv 1 \bmod {(1+i)^3} \\ \varpi | n l \prod^j_{i=1}r_i }} (1+N(\varpi)^{-1})^{-1}.
\end{align*}

 Now the arguments given in \cite[Section 6.2]{Young1}, \cite[Section 6.1]{Young2} and the proof of \cite[Lemma 4.4]{Sono} carry over to our case
with simple modifications to show that we have
\begin{equation*}
 D_N(k=0;s) = D_R(1, \cdots, 1;s).
\end{equation*}

  We then conclude that
\begin{align*}
 M_N(k=0) + M_R(1,\cdots,1) = \frac {\pi}{2} \sum_{\substack{ a \equiv 1 \bmod {(1+i)^3} \\(a,2l)=1}} \frac {\mu_{[i]}(a)}{N(a)^2}
\frac {1}{2\pi
   i} \int\limits\limits_{(2)} g_{\alpha_1, \cdots, \alpha_j}(s)
 \widehat{F}(1+\frac {js}2)D_N(k=0;s)  \frac {G_j(s)ds}{s}.
\end{align*}

  It follows from Lemma \ref{lemma:A} and Remark \ref{remark:zero} that we can move the contour of integration to $-1/2 + \varepsilon$ to cross a pole at $s=0$ only in the process.  The residue at $s=0$ gives the desired main term and the error term is easily estimated to be of the desired size.
\end{proof}

\subsection{Computing $M_N$: the term $M_N(k \neq 0)$}

    Using the expression given in \eqref{Fnexprssion} for $\widetilde{F_{n,j}}$, we see that
\begin{align*}
 M_N(k \neq 0) =& \frac{X}{2}  \sum_{\substack{ a \equiv 1 \bmod {(1+i)^3} \\(a,2l)=1 \\ N(a) \leq Y}} \frac{\mu_{[i]}(a)}{N(a)^2}  \frac{\sigma_{\alpha_1,\cdots,\alpha_j}(n)}{N(n)^{\thalf}}  \sum_{\substack {k \in
   \mz[i] \\ k \neq 0}}(-1)^{N(k)} \frac{g(k, ln)}{N(ln)}
\widetilde{F_{n,j}}\left(\sqrt{\frac {N(k)X}{2N(a^2nl)}}\right) \\
=& \frac{\pi}{2}  \sum_{\substack{ a \equiv 1 \bmod {(1+i)^3} \\(a,2l)=1 \\ N(a) \leq Y}} \frac{\mu_{[i]}(a)}{N(a)^2} \frac {1}{(2\pi i)^2}\int\limits\limits_{(c_u)}
\int\limits\limits_{(c_s)}\widehat{F}(1+u)
g_{\alpha_1, \cdots, \alpha_j}(s)  \left(\pi \left(\sqrt{\frac {1}{2N(a^2l)}}\right) \right )^{-(js-2u)}\frac{\Gamma \leg{js-2u}{2}}{\Gamma\leg{2-js+2u}{2}}  \\
& \times \sum_{\substack{ n \equiv 1 \bmod {(1+i)^3} \\ (n,a)=1}}
 \sum_{\substack {k \in
   \mz[i] \\ k \neq 0}}(-1)^{N(k)} \frac{\sigma_{\alpha_1,\cdots,\alpha_j}(n)}{N(n)^{\thalf +(1-j/2)s+u}}\frac{1}{N(k)^{js/2-u}} \frac{g(k, ln)}{N(ln)}du \frac {G_j(s)ds}{s},
\end{align*}
  where we set $c_s=c_u>3$ satisfying $c_s-c_u>2$.

  Now, we let $f(k)=g(k,n)/N(k)^s$ and we write $k = k_1k^2_2$ with $k_1$ square-free and $k_2 \in G$, where we recall here that $G$ is the set of generators of all ideals in $\mathcal{O}_K$ defined in Section \ref{sect: alybehv}. We break the sum over $k_1$ into two sums, depending on $(k_1, 1+i)=1$ or not, to get
\begin{align*}
   \sum_{\substack {k \in
   \mz[i] \\ k \neq 0}}(-1)^{N(k)} f(k)= & \sumstar_{\substack{k_1 \\ (k_1, 1+i) \neq 1}}\sum_{k_2}f(k_1k^2_2)+\sumstar_{\substack{k_1 \\ (k_1, 1+i) = 1}}\sum_{k_2}(-1)^{N(k_2)}f(k_1k^2_2)  \\
   = & \sumstar_{\substack{k_1 \\ (k_1, 1+i) \neq 1}}\sum_{k_2}f(k_1k^2_2)+\sumstar_{\substack{k_1 \\ (k_1, 1+i) = 1}}
   \left (2\sum_{k_2}f(2k_1k^2_2)-\sum_{k_2}f(k_1k^2_2) \right ),
\end{align*}
  where we note that $(1+i)$ is the only prime ideal in $\mathcal{O}_K$  that lies above the integral ideal $(2) \in \mz$.

    Note that when $(n, 1+i)=1$, $g(k,n)=g(2k,n)$ by Lemma \ref{Gausssum}. It follows that we have $f(2k_1k^2_2)=4^{-s}f(k_1k^2_2)$ so that
\begin{align*}
   \sum_{\substack {k \in
   \mz[i] \\ k \neq 0}}(-1)^{N(k)} f(k)= (2^{1-2s}-1)\sumstar_{\substack{k_1 \\ (k_1, 1+i) = 1}}
   \sum_{k_2}f(k_1k^2_2) +\sumstar_{\substack{k_1 \\ (k_1, 1+i) \neq 1}}\sum_{k_2}f(k_1k^2_2).
\end{align*}

    We apply the above expression to recast $M_N(k \neq 0)$ as
\begin{align*}
    M_N(k \neq 0)  &=\frac {\pi}{2}  \left (  \ \sumstar_{\substack{k_1 \\ (k_1, 1+i) = 1}} \frac {1}{N(k_1)^{js/2-u}}
 \mathcal{M}_{1}(s,u,k_1,l)+ \sumstar_{\substack{k_1 \\ (k_1, 1+i) \neq 1}}\frac {1}{N(k_1)^{js/2-u}}\mathcal{M}_{2}(s,u, k_1,l)\right ),
\end{align*}
    where
\begin{align}
\label{eq:MNpreDirichlet}
\begin{split}
 \mathcal{M}_{1}(s,u, k_1,l) =&  \sum_{\substack{ a \equiv 1 \bmod {(1+i)^3} \\(a,2l)=1 \\ N(a) \leq Y}} \frac{\mu_{[i]}(a)}{N(a)^2} \frac {1}{(2\pi i)^2}\int\limits\limits_{(c_u)}
\int\limits\limits_{(c_s)}\widehat{F}(1+u)
g_{\alpha_1, \cdots, \alpha_j}(s)  \left( \pi \left(\sqrt{\frac {1}{2N(a^2l)}}\right) \right )^{-(js-2u)}\\
&\times (2^{1-2(js/2-u)}-1)  \frac{\Gamma \leg{js-2u}{2}}{\Gamma\leg{2-js+2u}{2}}
J_{k_1,j}(js-2u, \frac 12+(1-\frac j2)s+u; l,a)du \frac {G_j(s)ds}{s},
\end{split}
\end{align}
and $J_{k_1,j}(v,w;l, a)$ is defined in \eqref{eq:J}. The formula for $\mathcal{M}_{2}(s,u,k_1,l)$ is identical to \eqref{eq:MNpreDirichlet} except that the factor $2^{1-2(js/2-u)} - 1$ is omitted.

  We move the contours to $c_s = \half + \varepsilon$ and $c_u =\frac {j}{4}-1$ retaining the relation $jc_s-2u>2$. In view of Lemma \ref{lemma:Jprop},  $J_{k_1,j}$ remains analytic in the process.  Again it follows from Lemma \ref{lemma:Jprop} and Remark \ref{remark:zero} that we cross poles of the Hecke $L$-functions at $u=-(1-\frac{j}{2})s - \alpha_i, 1 \leq i \leq j$ for $k_1 =\pm i$  only. For each $\alpha_i, 1 \leq i \leq j$, we denote $M_N(k_1 = \pm i, \alpha_i), 1 \leq i \leq 3$ for the contribution to $M_N(k_1 \neq 0)$ from the sums of the two residues corresponding to $k_1=\pm i$. Note further that by Lemma \ref{Gausssum} and \eqref{eq:J} that we have
$J_{i,j}(v,w;l, a)=J_{-i,j}(v,w;l, a)$ so that we shall denote $J_{j}(v,w;l, a)$ for $J_{i,j}(v,w;l, a)$ or $J_{-i,j}(v,w;l, a)$ from now on. Using this notation, we have
\begin{align}
\label{eq:MNksquarealpha}
\begin{split}
 M_N(k_1=\pm i, \alpha_i) = & \pi \sum_{\substack{ a \equiv 1 \bmod {(1+i)^3} \\(a,2l)=1 \\ N(a) \leq Y}} \frac{\mu_{[i]}(a)}{N(a)^2}
\frac {1}{2\pi i}
\int\limits\limits_{(c_s)}\widehat{F}(1 -(1-\frac{j}{2})s  - \alpha_i)
g_{\alpha_1, \cdots, \alpha_j}(s) \left( \pi \left(\sqrt{\frac {1}{2N(a^2l)}}\right) \right )^{-2(s + \alpha_i)}  \\
& \times (2^{1-2(s + \alpha_i)}-1)\frac{\Gamma (s+\alpha_i)}{\Gamma (1-s-\alpha_i)}
 \text{Res}_{w=\frac{1}{2} - \alpha_i} J_j(2s + 2\alpha_i, w;l,a) \frac {G_j(s)ds}{s} .
\end{split}
\end{align}

On the new lines of integration,  we argue as in Section 5.3 of \cite{Young2} using the following analogue estimation  of \cite[(4.1)]{G&Zhao3} for the second moment such that when $|\Re(\alpha)| \ll (\log{X})^{-1}$,
\begin{align*}
    \sumstar_{\substack{(d,2)=1 \\ N(d) \leq X }} \left| L_{2al}(\thalf+\alpha, \chi_{(1+i)^5d})  \right|^2 \ll N(al)^{\varepsilon}\left( X(1+|\Im(\alpha)|) \right)^{1+\varepsilon}
\end{align*}
  to see that
the sum over $k_1$ converges absolutely on these lines of integration and that with our choices of $c_u$ and $c_s$, the contribution to $M_N$ from these error terms is
\begin{equation*}
 \ll \sum_{N(a) \leq Y} N(a)^{-2} N(l a^2)^{1 + \varepsilon} N(l)^{-\half + \varepsilon} X^{j/4 + \varepsilon}
\ll N(l)^{1/2 + \varepsilon} Y X^{j/4 + \varepsilon}.
\end{equation*}

   We then conclude from the above that
\begin{align}
\label{MNknot0}
  M_N(k \neq 0) = \sum^j_{i=1}M_N(k_1=\pm i, \alpha_i)  +O(X^{j/4 + \varepsilon} Y N(l)^{1/2 + \varepsilon}).
\end{align}

\subsection{Computing $M_N$:  gathering terms}

   In this section we show that for any fixed $i$, the term $M_{\pm N}(k_1=\pm i, \alpha_i)$ combines naturally with the term $M_{\pm R}(\epsilon_1, \cdots, \epsilon_j)$, where we have $\epsilon_i=-1$ and $\epsilon_k=1$ for all $k \neq i$.  As a preparation, we first establish an Archimedean-type identity.
\begin{lemma}
\label{lemma:archcalc}
 Let $u$ be a complex number.  Then
\begin{align*}
 (2^{1-u} -1)\zeta_K(u) \frac {\Gamma(\frac {u}2)}{\Gamma(1-\frac {u}2)}
 = \frac {4}{\pi} \left ( \frac {\pi^2}{2} \right )^{u/2} \Gamma_{u/2}\zeta_{K,2}(1-u)
\end{align*}
where $\Gamma_u$ is defined by \eqref{gamma}.
\end{lemma}
\begin{proof}
   We use the functional equation \eqref{fcneqnL} for $\zeta_K(u)$  to see that
\begin{align*}
   \pi^{-u}\zeta_K(u) = \pi^{-(1-u)} \frac {\Gamma(1 -u)}{\Gamma(u)} \zeta_K(1-u).
\end{align*}

   Next, we apply the formula (see \cite[Formula 3, Section 8.335]{GR}):
\begin{align*}
 \Gamma(\frac {u}2)\Gamma (\frac {1+u}{2})=\frac {\sqrt{\pi}}{2^{u-1}}\Gamma(u)
\end{align*}
   to see that
\begin{align*}
  \frac {\Gamma(1 -u)}{\Gamma(u)}\frac {\Gamma(\frac {u}2)}{\Gamma(1-\frac {u}2)}
=\frac {2^{1-u}}{2^{u}}\frac {\Gamma(\frac {1-u}{2})}{\Gamma(\frac {1+u}{2})}.
\end{align*}

   Note also that
\begin{align*}
 (2^{1-u} -1) \zeta_K(1-u) = 2^{1-u} \zeta_{K,2}(1-u).
\end{align*}

  From this we obtain that
\begin{align*}
 (2^{1-u} -1)\zeta_K(u) \frac {\Gamma(\frac {u}2)}{\Gamma(1-\frac {u}2)} =& \pi^{-(1-2u)}\frac {2^{2(1-u)}}{2^u} \frac {\Gamma(\frac {1-u}{2})}{\Gamma(\frac {1+u}{2})}\zeta_{K,2}(1-u)
 = \frac {4}{\pi} \left ( \frac {\pi^2}{2} \right )^{u/2} \Gamma_{u/2}\zeta_{K,2}(1-u),
\end{align*}
  as desired.
\end{proof}

   Now we are ready to prove the next result.
\begin{lemma}
\label{lemma:MNk0andMR111}
For special choices of $G_j(s), 1 \leq j \leq 3$ described in Remark \ref{remark:zero}, we have
\begin{align}
\label{eq:MT1}
 & M_N(k=0) + M_{-N}(k_1=\pm i, \alpha) + M_{R}(-1) + M_{-R}(1) = \frac{\pi  \widehat{F}(1)}{2 \zeta_{K,2}(2)} N(l)^{-1/2}A_{\alpha}(l), \\
\label{eq:MNksquareandMR-11}
 & M_N(k_1=\pm i, \alpha) + M_R(-1,1)+ M_{-N}(k_1=\pm i, \beta) + M_{-R}(-1,1)= \pi A_{- \alpha, \beta}(l) \Gamma_{\alpha} \frac{ \widehat{F}(1- \alpha )}{2 \zeta_{K,2}(2) \sqrt{N(l_1)}},  \\
\label{eq:MNksquareandMR-111}
 & M_N(k_1=\pm i, \alpha) + M_R(-1,1,1) = \pi A_{- \alpha, \beta,\gamma}(l) \Gamma_{\alpha} \frac{ \widehat{F}(1- \alpha )}{2 \zeta_{K,2}(2) \sqrt{N(l_1)}} +
O(X^{3/4 + \varepsilon} N(l)^{\varepsilon}).
\end{align}
 The relations given in \eqref{eq:MNksquareandMR-11} and \eqref{eq:MNksquareandMR-111} are valid similarly if one replaces $\alpha$ by $\beta$ or by $\gamma$ (when $j=3$).
\end{lemma}
\begin{proof}
     We begin our proof in general by applying Lemma \ref{lemma:archcalc} with $u=2(s + \alpha_i)$ to \eqref{eq:MNksquarealpha}, thus obtaining
\begin{align*}
\begin{split}
 M_N(k_1=\pm i, \alpha_i) = & \frac {\pi}{2} \sum_{\substack{ a \equiv 1 \bmod {(1+i)^3} \\(a,2l)=1 \\ N(a) \leq Y}} \frac{\mu_{[i]}(a)}{N(a)^2}
\frac {1}{2\pi i}
\int\limits\limits_{(c_s)}\widehat{F}(1 -(1-\frac{j}{2})s  - \alpha_i)
g_{\alpha_1, \cdots, \alpha_j}(s) \left( \frac {\pi}{\sqrt{2}} \right )^{-2(s + \alpha_i)}  \\
& \times 2N(l a^2)^{s + \alpha_i}\Gamma_{s + \alpha_i}\zeta_{K,2}(1-2\alpha_i - 2s)
 \text{Res}_{w=\half - \alpha} \frac {4J_j(2s + 2\alpha_i, w;l,a)} {\pi\zeta_K(2s + 2\alpha_i)} \frac {G_j(s)ds}{s} .
\end{split}
\end{align*}

  Recall that $c_s = \half + \varepsilon$ and the residue of $\zeta_K(s)$ at $s=1$ equals $\pi/4$. We replace the residue of $\frac 4{\pi}J_j(2s + 2\alpha, w)$ at $w = \half - \alpha_i$ by the value of $J_j(2s + 2 \alpha_i, w)/\zeta_K(\half + w + \alpha_i)$ at $w = \half - \alpha_i$ to see that
\begin{align*}
 M_N(k_1=\pm i, \alpha_i) =& \frac {\pi}2\sum_{\substack{ a \equiv 1 \bmod {(1+i)^3} \\(a,2l)=1 \\ N(a) \leq Y}} \frac{\mu_{[i]}(a)}{N(a)^2}
\frac {1}{2\pi i}
\int\limits\limits_{(c_s)}\widehat{F}(1 -(1-\frac{j}{2})s   - \alpha)
g_{\alpha_1, \cdots, \alpha_j}(s)  \frac {G_j(s)}{s}
\\
& \times \Gamma_{s + \alpha_i} N(a)^{2 \alpha + 2s} D_N(k_1=\pm i, \alpha_i; s)
  ds,
\end{align*}
where
\begin{align*}
 D_N(k_1=\pm i, \alpha_i; s) =  2 N(l)^{s + \alpha_i}   \zeta_{K,2}(1-2\alpha_i - 2s)  \frac{J_j(2s + 2\alpha_i, w)}{ \zeta_K(2s + 2\alpha_i)
 \zeta_K(\half + w + \alpha_i)} \Big|_{w=\half - \alpha_i}.
\end{align*}

   Next, by writing $c=ab$ and setting $\epsilon_i=-1$ and $\epsilon_k=1$ for all $k \neq i$ in \eqref{eq:MRmainterms}, we deduce that
\begin{align*}
M_R(\epsilon_1, \cdots, \epsilon_j) =&  \frac{\pi }{2} \sum_{\substack{ a \equiv 1 \bmod {(1+i)^3} \\ (a,2l) = 1\\ N(a) > Y}}\frac{\mu_{[i]}(a)}{N(a)^2}
\frac{1}{2\pi i} \int\limits\limits_{(\varepsilon)}
\frac{G_j(s)}{s} g_{\alpha_1, \cdots,\alpha_j}(s)\widehat{F}(1 -(1-\frac{j}{2})s    - \alpha_i)
 \\
& \times \Gamma_{\alpha_i + s} N(a)^{2 \alpha_i + 2s}
D_R(\epsilon_1, \cdots, \epsilon_j;s)
  ds,
\end{align*}
where
\begin{align*}
 D_R(-1,1,1;s) =& \frac{1}{ \zeta_{K,2}(2)} \sum_{\substack{ b \equiv 1 \bmod {(1+i)^3} \\ (b,2l)=1}}  \frac{1}{N(b)^{2(1-\alpha_i-s)}} \\
 & \times \sum_{\substack{ r_i \equiv 1 \bmod {(1+i)^3} \\ 1 \leq i \leq 3 \\ r_i | ab}}
\prod^j_{i=1}\frac{\mu_{[i]}(r_i)}{N(r_i)^{\half + \alpha_i+s}}\frac{1}{\sqrt{N((l\prod^j_{i=1}r_i)^*)}}
 A_{\epsilon_1(\alpha_1+ s), \cdots, \epsilon_j(\alpha_j+ s)}(l\prod^j_{i=1}r_i)
 .
\end{align*}

   Using arguments similar to those used in the proof of \cite[Lemma 6.2]{Young2}, we see that
\begin{equation*}
 D_N(k_1=\pm i, \alpha_i; s) = D_R(\epsilon_1, \cdots, \epsilon_j;s).
\end{equation*}

   It follows from this that we have
\begin{align*}
M_N(k_1=\pm i, \alpha_i) + M_R(\epsilon_1, \cdots, \epsilon_j) =&  \frac{\pi }{2} \sum_{\substack{ a \equiv 1 \bmod {(1+i)^3} \\ (a,2l) = 1}}\frac{\mu_{[i]}(a)}{N(a)^2}
\frac{1}{2\pi i} \int\limits\limits_{(\varepsilon)}
\frac{G_j(s)}{s} g_{\alpha_1, \cdots,\alpha_j}(s)\widehat{F}(1 -(1-\frac{j}{2})s    - \alpha_i)
\\
& \times \Gamma_{\alpha_i + s} N(a)^{2 \alpha_i + 2s}
D_R(\epsilon_1, \cdots, \epsilon_j;s)
  ds.
\end{align*}
Grouping $ab$ into a variable and applying the M\"{o}bius formula implies that only $ab=1$ survives, which implies that $r_1 = r_2 = r_3 = 1$.  Thus
\begin{align}
\label{MN}
\begin{split}
 &M_N(k_1=\pm i, \alpha_i) + M_R(\epsilon_1, \cdots, \epsilon_j) \\
 = &
\frac{\pi}{2 \zeta_{K,2}(2) \sqrt{N(l_1)}}
\frac{1}{2\pi i} \int\limits\limits_{(\varepsilon)}
\frac{G_j(s)}{s} g_{\alpha_1, \cdots,\alpha_j}(s)\widehat{F}(1 -(1-\frac{j}{2})s    - \alpha_i)\Gamma_{\alpha_i + s}   A_{\epsilon_1(\alpha_1+ s), \cdots, \epsilon_j(\alpha_j+ s)}(l)
  ds.
\end{split}
\end{align}

  When $j=3$, we can move the contour of integration to $-1/2 + \varepsilon$, crossing a pole at $s=0$ only, in view of Lemma \ref{lemma:A} and Remark \ref{remark:zero}.  The residue at $s=0$ gives the main term in \eqref{eq:MNksquareandMR-111}, and the error term is easily  seen to be of the desired size.

  For $j=1,2$, we further obtain an expression for $M_{-N}(k_1=\pm i, \alpha_i) + M_{-R}(\epsilon_1, \cdots, \epsilon_j)$ from the above expression using Remark \ref{remark:M1toM2}, where $\epsilon_1, \cdots, \epsilon_j$ are the same as those in \eqref{MN}. Using the relation that $\widehat{F_{-(\alpha_j)}}(w)=\widehat{F}(w-\sum^j_{i=1}\alpha_i)$, we see that
\begin{align*}
& M_{-N}(k_1=\pm i, \alpha_i) + M_{-R}(\epsilon_1, \cdots, \epsilon_j)
\\
= & \frac{\pi}{2 \zeta_{K,2}(2) \sqrt{N(l_1)}}
\frac{1}{2\pi i} \int\limits\limits_{(\varepsilon)}
\frac{G_j(s)}{s} g_{-\alpha_1, \cdots,-\alpha_j}(s)\widehat{F}(1 -(1-\frac{j}{2})s    +\alpha_i-\sum^j_{i=1}\alpha_i)\Gamma_{-\alpha_i + s}   \Gamma_{(\alpha_j)}  A_{\epsilon_1(-\alpha_1+ s), \cdots, \epsilon_j(-\alpha_j+ s)}(l)
  ds.
\end{align*}

   We apply a change of variable $s \rightarrow -s$ to recast the above as
\begin{align}
\label{MNminus}
\begin{split}
& M_{-N}(k_1=\pm i, \alpha_i) + M_{-R}(\epsilon_1, \cdots, \epsilon_j)  \\
=&-\frac{\pi}{2 \zeta_{K,2}(2) \sqrt{N(l_1)}}
\frac{1}{2\pi i} \int\limits\limits_{(-\varepsilon)}
\frac{G_j(s)}{s} g_{-\alpha_1, \cdots,-\alpha_j}(-s)\widehat{F}(1 +(1-\frac{j}{2})s    +\alpha_i-\sum^j_{i=1}\alpha_i)\Gamma_{-\alpha_i - s}\Gamma_{(\alpha_j)}    A_{\epsilon_1(-\alpha_1- s), \cdots, \epsilon_j(-\alpha_j- s)}(l)
  ds.
\end{split}
\end{align}

   We now deduce from the identity
\begin{equation*}
 g_{-\alpha}(-s)\Gamma_{-\alpha-s} \Gamma_{\alpha} = g_{\alpha}(s)
\end{equation*}
   and the identity $\Gamma_{\alpha}=\Gamma^{-1}_{-\alpha}$ that
\begin{equation*}
 g_{-\alpha_1, \cdots,-\alpha_j}(-s)\Gamma_{-\alpha_i - s}\Gamma_{(\alpha_j)} =g_{\alpha_1, \cdots,\alpha_j}(s)\frac {\Gamma_{(\alpha_j+s)}}{\Gamma_{\alpha_i + s}}.
\end{equation*}

  When $j=2$, the above allows us to see that the two integrands on the right-hand sides of \eqref{MN} and \eqref{MNminus} (with $a_i$ replaced by $a_{j-i+1}$) are negative to each other, hence the sum of the two integrals equals to the residue at $s=0$ of the integrand in \eqref{MN}, thus proving \eqref{eq:MNksquareandMR-11}. Applying the above discussions similarly to the case $j=1$ by taking note of \eqref{eq:right} allows us to establish \eqref{eq:MT1} as well.
\end{proof}

\subsection{Completion of the proof}

  We are now able to complete the proof of Theorem \ref{theo:recursive2}. We first consider the case $j=1$. In this case, we note that it follows from Remark \ref{remark:M1toM2} that we also have
\begin{equation}
\label{eq:MT2}
 M_{-N}(k=0) + M_{N}(k_1=\pm i, \alpha) + M_{R}(-1) + M_{R}(1) =   \Gamma_{\alpha}\frac{\pi \widehat{F}(1-\alpha)}{2 \zeta_{K,2}(2)\sqrt{N(l)}}A_{-\alpha}(l).
\end{equation}

  Combining Lemma \ref{lemma:MRresult1}, \eqref{MNknot0} and taking note of Remark \ref{remark:M1toM2}, we get
\begin{align}
\label{thm:combine}
\begin{split}
 M_{\alpha}(l) =& M_N(k=0) + M_{-N}(k = 0) + M_N(k_1 = \pm i, \alpha) + M_{-N}(k_1 = \pm i, \alpha)\\
\\
& + M_{R}(1) + M_{R}(-1) + M_{-R}(1) + M_{-R}(-1) + O\lp \frac{X^{f + \varepsilon}}{Y^{2f-1}} N(l)^{1/2 + \varepsilon} + X^{1/4 + \varepsilon} Y N(l)^{1/2 + \varepsilon} +X^{1/2 + \varepsilon} \rp.
\end{split}
\end{align}
   Now applying \eqref{eq:MT1} and \eqref{eq:MT2} in the above expression and setting $Y = X^{\frac{1}{4}}$ in \eqref{thm:combine} allows us to see that the statement of Theorem \ref{theo:recursive2} is valid for $j=1$.

   For the case $j=2$, we combine Lemma \ref{lemma:MRresult1} and \eqref{MNknot0} and Remark \ref{remark:M1toM2} to obtain
\begin{align}
\label{thm:combine2}
\begin{split}
 M_{\alpha, \beta}(l) =& M_N(k=0) + M_R(1, 1)+ M_{-N}(k=0) + M_{-R}(1, 1)+M_{-N}(k = 0) \\
& +M_N(k_1=\pm i, \alpha) + M_R(-1,1)+ M_{-N}(k_1=\pm i, \beta) + M_{-R}(-1,1)\\
& +M_N(k_1=\pm i, \beta) + M_R(1,-1)+ M_{-N}(k_1=\pm i, \alpha) + M_{-R}(1,-1) \\
& + O\lp \frac{X^{f + \varepsilon}}{Y^{2f-1}} N(l)^{1/2 + \varepsilon} + X^{1/2 + \varepsilon} Y N(l)^{1/2 + \varepsilon} \rp.
\end{split}
\end{align}
   Now applying \eqref{eq:MNk0andMR111} and \eqref{eq:MNksquareandMR-11} in the above expression as well as Remark \ref{remark:M1toM2} and setting $Y = X^{\frac{2f-1}{4f}}$ in \eqref{thm:combine2} allows us to see that the statement of Theorem \ref{theo:recursive2} is valid for $j=2$.

  For the case $j=3$, we combine Lemma \ref{lemma:MRresult1} and \eqref{MNknot0} to see that
\begin{align*}
\begin{split}
 M_1 =& M_N(k=0) + M_N(k_1=\pm i, \alpha) + M_N(k_1=\pm i,  \beta) + M_N(k_1=\pm i, \gamma) \\
& +M_R(1,1,1) + M_R(1,1,-1) + M_R(1,-1, 1)+M_R(-1,1,1) \\
&+ O\lp \frac{X^{f + \varepsilon}}{Y^{2f-1}} N(l)^{1/2 + \varepsilon} + X^{3/4 + \varepsilon} Y N(l)^{1/2 + \varepsilon} \rp.
\end{split}
\end{align*}

   We now apply \eqref{eq:MNk0andMR111} and \eqref{eq:MNksquareandMR-111} to recast the above as
\begin{align}
\label{thm:M1}
\begin{split}
  M_1  =& \pi A_{\alpha, \beta, \gamma}(l) \frac{\widehat{F}(1)}{2 \zeta_{K,2}(2) \sqrt{N(l_1)}}  \\
& + \pi A_{- \alpha, \beta,\gamma}(l) \Gamma_{\alpha} \frac{ \widehat{F}(1- \alpha )}{2 \zeta_{K,2}(2) \sqrt{N(l_1)}} + \pi A_{\alpha, -\beta,\gamma}(l) \Gamma_{\beta} \frac{ \widehat{F}(1- \beta )}{2 \zeta_{K,2}(2) \sqrt{N(l_1)}}+ \pi A_{ \alpha, \beta,-\gamma}(l) \Gamma_{\gamma} \frac{ \widehat{F}(1- \gamma )}{2 \zeta_{K,2}(2) \sqrt{N(l_1)}}\\
&+ O\lp \frac{X^{f + \varepsilon}}{Y^{2f-1}} N(l)^{1/2 + \varepsilon} + X^{3/4 + \varepsilon} Y N(l)^{1/2 + \varepsilon} \rp.
\end{split}
\end{align}

We then obtain an asymptotic for $M_{-1}$ using Remark \ref{remark:M1toM2}, which gives the remaining four main terms in \eqref{eq:Malpha3} plus the same error as given in \eqref{thm:M1}.  We now readily deduce the assertion of Theorem \ref{theo:recursive2} for $j=3$ by setting $Y = X^{\frac{f-\frac34}{2f}}$. This completes the proof of Theorem \ref{theo:recursive2}.

\section{Proof of Theorem \ref{thm: nonvanishing}}
\label{sect: nonvanishing}

   We consider the following mollifier
\begin{align}
\label{Md}
  M(d) =\sum_{\substack{ l \equiv 1 \bmod {(1+i)^3} \\ N(l) \leq M}} \lambda(l)\sqrt{N(l)}\chi_{(1+i)^5d}(l).
\end{align}
  Our goal is to choose $\lambda(l)$ optimally such that the following mollified first and second  moments (corresponding to $j=1,2$, respectively) are comparable:
\begin{align*}
  S(L(\half, \chi_{(1+i)^5d})^jM(d)^j; \Phi)=\frac 1X \sum_{\substack{ d \in \mathcal{O}_K \\ (d, 2)=1}}\mu^2_{[i]}(d)L(\half, \chi_{(1+i)^5d})^jM(d)^j\Phi(\frac {N(d)}{X}).
\end{align*}
  Here we set $M = (\sqrt{X})^{\theta}$ for some $\theta < 1 - \varepsilon$ and $\Phi$ is given in Theorem \ref{theo:mainthm} such that we take $\Phi$ to be an approximation to the characteristic function of $(1, 2)$ so that $\widehat{\Phi}(1) \sim 1$.
  To specify $\lambda(l)$, we first make a linear change of variables to define for primary $\gamma$,
\begin{align*}
 \xi(\gamma)=\sum_{\substack{a \equiv 1 \bmod {(1+i)^3}}}\frac {\lambda(a\gamma)}{h(a)}\frac {N(a) d_{[i]}(a)}{\sigma_{[i]}(a)}.
\end{align*}

   Note here that we can recover $\lambda$ from $\xi$ by the following relation:
\begin{align}
\label{eq:lambda}
 \lambda(l)=\sum_{\substack{a \equiv 1 \bmod {(1+i)^3}}}\frac {\mu_{[i]}(a)}{h(a)}\frac {N(a) d_{[i]}(a)}{\sigma_{[i]}(a)}\xi(la).
\end{align}

   Thus, in order to determine $\lambda(l)$, it suffices to define $\xi(\gamma)$. We shall assume that $\xi(\gamma)$ is supported on primary square-free elements $\gamma $ satisfying $N(\gamma) \leq M$. We then note that \eqref{eq:lambda} implies that $\lambda(l)$  is also supported on primary square-free elements $\gamma $ satisfying $N(\gamma) \leq M$.

   We shall further require that
\begin{align}
\label{eq:xibound}
 |\xi(\gamma)|=\frac 1{N(\gamma)\log^2 M}\prod_{\substack{\varpi \equiv 1 \bmod {(1+i)^3} \\ \varpi | \gamma}}\lp1+O\lp \frac {1}{N(\varpi)} \rp \rp.
\end{align}
   It is then easy to deduce from this and \eqref{eq:lambda} that $\lambda(l) \ll N(l)^{-1+\varepsilon}$.

\subsection{First mollified moment}
   Our evaluation of the first mollified moment requires us to evaluate $M_0(l)$ explicitly, where $M_0(l)$ is defined in \eqref{Malphal}. This can be done directly from Theorem \ref{thm: Malphal1} by considering the limit as $\alpha \rightarrow 0$ of the asymptotic expression given in \eqref{eq:Malphal} for $M_{\alpha}(l)$ (with $f=1/2$ there). In this way, we obtain the following result analogue to \cite[Proposition 1.2]{sound1}:
\begin{theorem}
\label{theo:1stmoment}
Let $\Phi$ be given in Theorem \ref{theo:mainthm}.   For any primary square-free  $l \in \mathcal{O}_K$ and any $\varepsilon>0$, we have
\begin{align*}
 \sumstar_{(d,2)=1} L(\thalf, \chi_{(1+i)^5d}) \Phi\leg{N(d)}{X}\chi_{(1+i)^5d}(l) =&  \frac {\pi^2}{4} \frac{ \widehat{\Phi}(1)X}{\zeta_{K}(2)\sqrt{N(l)}}\frac {C}{g(l)}\left (\log\frac {\sqrt{X}}{N(l)}+C_2+\sum_{\substack{\varpi \equiv 1 \bmod {(1+i)^3} \\ \varpi | l}} \frac {C_2(\varpi)\log N(\varpi)}{N(\varpi)} \right ) \\
&+O(N(l)^{1/2 + \varepsilon} X^{\frac 12 + \varepsilon}),
\end{align*}
  where
\begin{align*}
C=\frac 1{3}\prod_{\substack{\varpi \equiv 1 \bmod {(1+i)^3}}}\left (1-\frac {1}{N(\varpi)(N(\varpi)+1)} \right ),  \quad
g(l)=\prod_{\substack{\varpi \equiv 1 \bmod {(1+i)^3} \\ \varpi | l}}\lp \frac {N(\varpi)+1)}{N(\varpi)}\rp \left (1-\frac {1}{N(\varpi)(N(\varpi)+1)} \right ).
\end{align*}
  Moreover, $C_2$ is a constant depending only on $\Phi$ and $C_2(\varpi) \ll 1$ for all $\varpi$.
\end{theorem}

  We then apply \eqref{Md} and Theorem \ref{theo:1stmoment} to see that
\begin{align*}
  S(L(\half, \chi_{(1+i)^5d})M(d); \Phi) =&  \frac {\pi^2}{4} \frac{ C\widehat{\Phi}(1)}{\zeta_{K}(2)}\sum_{\substack{l \equiv 1 \bmod {(1+i)^3} \\ N(l) \leq M}}\frac {\lambda(l)}{g(l)}\left (\log\frac {\sqrt{X}}{N(l)}+C_2+\sum_{\substack{\varpi \equiv 1 \bmod {(1+i)^3} \\ \varpi | l}} \frac {C_2(\varpi)}{N(\varpi)\log N(\varpi)} \right ) \\
&+O(X^{-\varepsilon}).
\end{align*}

   We now define a multiplicative function $g_1(\gamma)$ on primary, square-free $\gamma$ such that for any primary prime $\varpi$, we have
\begin{align*}
g_1(\varpi) =\frac 1{g(\varpi)}-\frac {2N(\varpi)}{h(\varpi)(N(\varpi)+1)}.
\end{align*}
  We note that $g_1(\varpi)=-1+O(1/N(\varpi))$.  Using \eqref{eq:lambda} to write $\lambda$ in terms of $\xi$, we derive that
\begin{align*}
 \sum_{\substack{l \equiv 1 \bmod {(1+i)^3} \\ N(l) \leq M}}\frac {\lambda(l)}{g(l)}\log\frac {\sqrt{X}}{N(l)}=&
\sum_{\substack{\gamma \equiv 1 \bmod {(1+i)^3}}}g_1(\gamma)\left (\log (\sqrt{X}N(\gamma))+O(\sum_{\substack{\varpi \equiv 1 \bmod {(1+i)^3} \\ \varpi | \gamma}} \frac {\log N(\varpi)}{N(\varpi)}) \right ) \\
=& \sum_{\substack{\gamma \equiv 1 \bmod {(1+i)^3}}}g_1(\gamma)\left (\log (\sqrt{X}N(\gamma))\right )+O\left(\frac {1}{\log X} \right ),
\end{align*}
  where the last estimation above follows from \eqref{eq:xibound}.

  Similar arguments imply that
\begin{align*}
 \sum_{\substack{l \equiv 1 \bmod {(1+i)^3} \\ N(l) \leq M}}\frac {\lambda(l)}{g(l)} \left ( C_2+\sum_{\substack{\varpi \equiv 1 \bmod {(1+i)^3} \\ \varpi | l}} \frac {C_2(\varpi)}{N(\varpi)\log N(\varpi)} \right ) \ll \frac {1}{\log X}.
\end{align*}

   We then conclude from the above discussions that the first mollified moment is
\begin{align}
\label{1stmollifiedmoment}
 S(L(\half, \chi_{(1+i)^5d})M(d); \Phi)= \frac {\pi^2}{4}  \frac{ C\widehat{\Phi}(1)}{\zeta_{K}(2)}\sum_{\substack{\gamma \equiv 1 \bmod {(1+i)^3}}}g_1(\gamma)\left (\log \left(\sqrt{X}N(\gamma) \right )\right )+O\left(\frac {1}{\log X} \right ).
\end{align}

\subsection{Second mollified moment}

   To evaluate the second mollified moment, we shall not apply an approach similar to our treatment for the first mollified moment since the error term in the asymptotic expression for $M_{\alpha, \beta}(l)$ given in Theorem \ref{thm: Malphal1} is too large in the $l$ aspect (of size $N(l)^{1/2 + \varepsilon}$). This would not allow us to take $\theta$ to be close to $1$. Rather, we follow the approach of Soundararajan in \cite{sound1} here.

    Let $Y$ be a parameter and we write $\mu_{[i]}^2(d)=M_Y(d)+R_Y(d)$ where
\begin{equation*}
    M_Y(d)=\sum_{\substack {l^2|d \\ N(l) \leq Y}}\mu_{[i]}(l) \; \quad \mbox{and} \; \quad  R_Y(d)=\sum_{\substack {l^2|d \\ N(l) >
    Y}}\mu_{[i]}(l).
\end{equation*}

   We then have
\begin{align*}
  S(L(\half, \chi_{(1+i)^5d})^2M(d)^2; \Phi)=S_M(L(\half, \chi_{(1+i)^5d})^2M(d)^2; \Phi)+S_R(L(\half, \chi_{(1+i)^5d})^2M(d)^2; \Phi),
\end{align*}
  where
\begin{align*}
 S_M(L(\half, \chi_{(1+i)^5d})^2M(d)^2; \Phi)=& \frac 1X \sum_{\substack{ d \in \mathcal{O}_K \\ (d, 2)=1}}M_Y(d)L(\half, \chi_{(1+i)^5d})^jM(d)^j\Phi(\frac {N(d)}{X}), \\
  S_R(L(\half, \chi_{(1+i)^5d})^2M(d)^2; \Phi)=& \frac 1X \sum_{\substack{ d \in \mathcal{O}_K \\ (d, 2)=1}}R_Y(d)L(\half, \chi_{(1+i)^5d})^jM(d)^j\Phi(\frac {N(d)}{X}).
\end{align*}

   Similar to \cite[Proposition 1.1]{sound1}, we can show that when $N(l)\ll N(1)^{-1+\varepsilon}$,
\begin{align}
\label{SL}
   S_R(L(\half, \chi_{(1+i)^5d})^2M(d)^2; \Phi) \ll \frac {X^{\varepsilon}}{Y}+\frac {M^{j/2}}{X^{1/2-\epsilon}}.
\end{align}

  To evaluate $S_M(L(\half, \chi_{(1+i)^5d})^2M(d)^2; \Phi)$, we introduce two notations now. First, we denote for any integer $j \geq 0$,
\begin{align*}
  \Phi_{(j)}=\max_{0 \leq i \leq j}\int\limits_{\mr}|\Phi^{(i)}(t)|dt.
\end{align*}

  Secondly, for all integers $j > 0$,  we define $\Lambda_j(n)$ to be the function defined on integral ideals of $K$ which equals the coefficient of $N(n)^{-s}$ in the Dirichlet series expansion of $(-1)^{j}\zeta^{(j)}_K(s)/\zeta_K(s)$. In particular, $\Lambda_1(n)$ is the usual von Mangoldt function $\Lambda(n)$ on $K$. We note that $\Lambda_j(n)$ is supported on elements $n$ in $\mathcal{O}_{K}$ such that $((n))$ has at most $j$ distinct prime ideal factors, and $\Lambda_j (n) \ll_j (\log N(n))^j$.

   Now, we are ready to state our result on $S_M(L(\half, \chi_{(1+i)^5d})^2M(d)^2; \Phi)$. We omit its proof here since it is similar to that of \cite[Proposition 1.2]{sound1}. We only point out here the that triple pole of $\zeta_K(1+2s)^3$ at $s=0$ contributes a factor of $(\pi/4)^3$. One can also derive the main term given in \eqref{eq:2ndmoment} below from $M_{\alpha, \beta}(l)$ defined in \eqref{Malphal} using Lemma 2.3 in \cite{Sono}.
\begin{theorem}
\label{theo:2ndmoment}
Let $\Phi$ be given in Theorem \ref{theo:mainthm}.  For any primary $l \in \mathcal{O}_K$ such that $l=l_1l^2_2$ such that $l_1$ is primary and square-free, we have for any $\varepsilon>0$,
\begin{align}
\label{eq:2ndmoment}
\begin{split}
& S_M(L(\half, \chi_{(1+i)^5d})^2M(d)^2; \Phi) \\
=&  \frac {\pi^4}{4^3} \frac{D \widehat{\Phi}(1)}{36\zeta_{K}(2)}\frac {d(l_1)}{\sqrt{N(l)}}\frac {N(l_1)}{\sigma_{[i]}(l_1)h(l)}\Big (\log^3\frac {X}{N(l_1)}-3\sum_{\substack{\varpi \equiv 1 \bmod {(1+i)^3} \\ \varpi | l}} \log^2 N(\varpi)\log \frac {X}{N(l_1)}+O(l) \Big ) \\
&+O \left(\Phi_{(2)}\Phi^{\epsilon}_{(3)}\frac {N(l)^{\half+\varepsilon} Y^{1+\varepsilon}}{X^{\half+\varepsilon}}+ \frac {N(l)^{\varepsilon} X^{\varepsilon}}{\sqrt{N(l_1)}Y}+\frac {N(l)^{\varepsilon} X^{\varepsilon}}{(N(l_1)X)^{1/4}} \right ),
\end{split}
\end{align}
  where $h$ is the multiplicative function defined on primary prime powers by
\begin{align*}
h(\varpi^k)=1+\frac 1{N(\varpi)}+\frac 1{N(\varpi)^2}-\frac 4{N(\varpi)(N(\varpi)+1)}, \quad (k \geq 1)
\end{align*}
  and
\begin{align*}
 D=\frac 18\prod_{\substack{\varpi \equiv 1 \bmod {(1+i)^3}}}\left (1-\frac 1{N(\varpi)} \right )h(\varpi).
\end{align*}
  Also,
\begin{align*}
 O(l)=& \sum^3_{j,k=0}\sum_{\substack{m \equiv 1 \bmod {(1+i)^3}\\ m |l_1}}\sum_{\substack{n \equiv 1 \bmod {(1+i)^3}\\ n |l_1}}
\frac {\Lambda_j(m)}{N(m)}\frac {\Lambda_k(n)}{N(n)}D(m,n)Q_{j,k}(\log \frac {X}{N(l_1)}) \\
& -3(A+B\frac {\widehat{\Phi}'(1)}{\widehat{\Phi}(1)})\sum_{\substack{\varpi \equiv 1 \bmod {(1+i)^3} \\ \varpi | l}} \log^2 N(\varpi).
\end{align*}
  where $A$ and $B$ are absolute constants and $D(m, n) \ll 1$ uniformly for all $m$ and $n$. The $Q_{j,k}$ are polynomials of degree $\leq 2$ whose coefficients involve only absolute constants and linear combinations of $\frac {\widehat{\Phi}^{(j)}(1)}{\widehat{\Phi}(1)}$ for $1 \leq j \leq 3$.
\end{theorem}

  Combining \eqref{SL} and \eqref{eq:2ndmoment} and setting $Y=X^{\varepsilon}$, we see that
\begin{align*}
 & S(L(\half, \chi_{(1+i)^5d})^2M(d)^2; \Phi) \\
= & \frac {\pi^4}{4^3}  \frac{ D\widehat{\Phi}(1)}{36\zeta_{K}(2)}\sum_{\substack{l \equiv 1 \bmod {(1+i)^3}}}\left ( \sum_{\substack{r,s \equiv 1 \bmod {(1+i)^3} \\ rs=l}} \lambda(r)\lambda(s)\right )\frac {\sqrt{N(l)}}{h(l)}\frac {d_{[i]}(l_1)}{\sqrt{N(l_1)})} \frac {N(l_1)}{\sigma_{[i]}(l_1)} \\
& \times \left ( \log^3\frac {X}{N(l_1)}-3\sum_{\substack{\varpi \equiv 1 \bmod {(1+i)^3} \\ \varpi |l_1}}\log^2 N(\varpi)\log \frac {X}{N(l_1)}+O(l) \right )+O(X^{-\varepsilon}).
\end{align*}

  We write $r = a\alpha$ and $s = b\alpha$ where $a$ and $b$ are co-prime primary elements. As $\lambda$ is assumed to be supported on square-free elements, we deduce that $\alpha = l_2$ and $l_1 = ab$. Thus we obtain from the above that
\begin{align*}
 & S(L(\half, \chi_{(1+i)^5d})^2M(d)^2; \Phi) \\
= & \frac {\pi^4}{4^3} \frac{ D\widehat{\Phi}(1)}{36\zeta_{K}(2)}\sum_{\substack{\alpha \equiv 1 \bmod {(1+i)^3}}} \frac {N(\alpha)}{h(\alpha)}\sum_{\substack{a,b \equiv 1 \bmod {(1+i)^3} \\ (a,b)=1}}\frac {\lambda(a\alpha)}{h(a)}\frac {\lambda(b\alpha)}{h(b)}\frac {ad_{[i]}(a)}{\sigma_{[i]}(a)} \frac {b d_{[i]}(a)}{\sigma_{[i]}(b)} \\
& \times \left ( \log^3\frac {X}{N(ab)}-3\sum_{\substack{\varpi \equiv 1 \bmod {(1+i)^3} \\ \varpi |ab}}\log^2 N(\varpi)\log \frac {X}{N(ab)}+O(\alpha^2 ab) \right )+O(X^{-\varepsilon}).
\end{align*}

  Using the M\"obius function to remove the condition that $(a,b)=1$, we see that
\begin{align}
\label{secmoment0}
\begin{split}
 &  S(L(\half, \chi_{(1+i)^5d})^2M(d)^2; \Phi) \\
= & \frac {\pi^4}{4^3} \frac{ D\widehat{\Phi}(1)}{36\zeta_{K}(2)}\sum_{\substack{\alpha \equiv 1 \bmod {(1+i)^3}}} \frac {N(\alpha)}{h(\alpha)}\sum_{\substack{\beta \equiv 1 \bmod {(1+i)^3}}}\frac {\mu_{[i]}(\beta)}{h(\beta)^2} \frac {\beta^2 d_{[i]}(\beta)^2}{\sigma_{[i]}(\beta)^2}\sum_{\substack{a,b \equiv 1 \bmod {(1+i)^3} }}\frac {\lambda(a\alpha\beta)}{h(a)}\frac {\lambda(b\alpha\beta)}{h(b)}\frac {ad_{[i]}(a)}{\sigma_{[i]}(a)} \frac {b d_{[i]}(a)}{\sigma_{[i]}(b)} \\
& \times \left ( \log^3\frac {X}{N(ab\beta^2)}-3\sum_{\substack{\varpi \equiv 1 \bmod {(1+i)^3} \\ \varpi |ab\beta}}\log^2 N(\varpi)\log \frac {X}{N(ab\beta^2)}+O(\alpha^2\beta^2 ab) \right )+O(X^{-\varepsilon}).
\end{split}
\end{align}

   We further define a multiplicative function $H(n)$ on primary, square-free $n$ such that for any primary prime $\varpi$,
\begin{align*}
 H(\varpi)=1-\frac {4N(\varpi)}{h(\varpi)(N(\varpi)+1)^2}=1+O(\frac 1{N(\varpi)}).
\end{align*}

   By setting $\gamma=\alpha\beta$ in \eqref{secmoment0} and proceeding similarly to the arguments in Section 6.2 of \cite{sound1}, we deduce that the second mollified moment is
\begin{align}
\label{secmoment}
\begin{split}
 &  S(L(\half, \chi_{(1+i)^5d})^2M(d)^2; \Phi) \\
= & \frac {\pi^4}{4^3}  \frac{ D\widehat{\Phi}(1)}{36\zeta_{K}(2)}\sum_{\substack{\gamma \equiv 1 \bmod {(1+i)^3} }}\frac {N(\gamma)H(\gamma)}{h(\gamma)}\sum_{\substack{a,b \equiv 1 \bmod {(1+i)^3} \\ (a,b)=1}}\frac {\lambda(a\gamma)}{h(a)}\frac {\lambda(b\gamma)}{h(b)}\frac {ad_{[i]}(a)}{\sigma_{[i]}(a)} \frac {b d_{[i]}(a)}{\sigma_{[i]}(b)} \\
& \times \Big ( \log^3\frac {X}{N(ab)}-3 \log \frac {X}{N(ab)}\Big ( \sum_{\substack{\varpi \equiv 1 \bmod {(1+i)^3} \\ \varpi |a}}\log^2 N(\varpi)+ \sum_{\substack{\varpi \equiv 1 \bmod {(1+i)^3} \\ \varpi |b}}\log^2 N(\varpi) \Big )\Big )+O(\frac 1{\log X}).
\end{split}
\end{align}

\subsection{Optimizing the mollified moments}

  It follows from \eqref{secmoment} that the second mollified moment looks like
\begin{align}
\label{secondmollifiedmoment0}
 & \frac {\pi^4}{4^3} \frac{ D\widehat{\Phi}(1)}{36\zeta_{K}(2)}\log^3 X \sum_{\substack{\gamma \equiv 1 \bmod {(1+i)^3} }}\frac {N(\gamma)H(\gamma)}{h(\gamma)}\xi(\gamma)^2.
\end{align}

 As the above is a diagonal quadratic form of $\xi(\gamma)$, we see that in order to choose a mollifier
to minimize \eqref{secondmollifiedmoment0} for fixed \eqref{1stmollifiedmoment}, we need to choose $\xi(\gamma)$
so that it is proportional to
\begin{align*}
 & \frac {h(\gamma)g_1(\gamma)}{N(\gamma)H(\gamma)}\log (\sqrt{X}\gamma).
\end{align*}

   We shall here follow the choice made in \cite[(6.8)]{sound1} to choose for primary square-free $\gamma \leq M$ such that
\begin{align*}
  \xi(\gamma)=\frac {C}{D\log^3 M}\frac {h(\gamma)g_1(\gamma)}{N(\gamma)H(\gamma)}\log (\sqrt{X}\gamma).
\end{align*}
   We notice that the above choice of $\xi$ does satisfy the condition \eqref{eq:xibound}.

   Similar to \cite[(6.8)]{sound1}, we see have that (keeping in mind that the residue of $\zeta_K(s)$ at $s=1$ is $\pi/4$)
\begin{align}
\label{elemargm}
\begin{split}
 &  \frac {C^2}{D \log^3 M}\sum_{\substack{\gamma \equiv 1 \bmod {(1+i)^3} \\ N(\gamma) \leq x }}\mu^2_{[i]}((1+i)\gamma)\frac {h(\gamma)g_1(\gamma)^2}{N(\gamma)H(\gamma)} \\
=& \frac \pi{4} \frac {C^2}{2D}\prod_{\substack{\varpi \equiv 1 \bmod {(1+i)^3} \\ }}\left (1-\frac 1{N(\varpi)} \right )\left (1+  \frac {h(\varpi)g_1(\varpi)^2}{N(\varpi)H(\varpi)}\right )(\log (X) +O(1))\\
=& \frac \pi{4}\frac 49(\log (X) +O(1)).
\end{split}
\end{align}

  We apply \eqref{elemargm} to \eqref{1stmollifiedmoment} via partial summation to see that the first mollified moment is
\begin{align}
\label{1stmollifiedmom}
\begin{split}
 S(L(\half, \chi_{(1+i)^5d})M(d); \Phi)\sim & \frac {\pi^2}{4} \frac {C^2\widehat{\Phi}(1)}{D\zeta_{K}(2)\log^3 M}\sum_{\substack{\gamma \equiv 1 \bmod {(1+i)^3} \\ N(\gamma) \leq M }}\mu^2_{[i]}((1+i)\gamma)\frac {h(\gamma)g_1(\gamma)^2}{N(\gamma)H(\gamma)}\log^2 (\sqrt{X}\gamma) \\
\sim & \left ( \frac \pi{4} \right )^2 \frac 29 \left ( \left (1+\frac 1{\theta} \right )^3-\frac 1{\theta^3} \right ) \frac {2\pi\widehat{\Phi}(1)}{3\zeta_{K}(2)}.
\end{split}
\end{align}

  Now, we proceed to evaluate the second mollified moment for the chosen $\xi$. For this, we define for rational integers $j \geq 0$,
\begin{align*}
 & \xi_j(\gamma)= \sum_{\substack{a \equiv 1 \bmod {(1+i)^3}}}\frac {\lambda(a\gamma)}{h(a)}\frac {d_{[i]}(a)}{\sigma_{[i]}(a)} (\log N(a))^j.
\end{align*}

  Similar to \cite[(6.11a)-(6.11c)]{sound1}, we see that for primary square-free element $\gamma$ satisfying $N(\gamma) \leq M$, we have
\begin{align*}
\begin{split}
 \xi_1(\gamma)=& -\frac {C}{D \log^3 M}\frac {h(\gamma)g_1(\gamma)}{N(\gamma)H(\gamma)}\Big ( 2\log \frac {M}{N(\gamma)}\log (\sqrt{X}N(\gamma))+\log^2 \frac {M}{N(\gamma)}+O\Big(\log M(1+ \sum_{\substack{q \equiv 1 \bmod {(1+i)^3} \\ q | \gamma}}\frac {\log N(q)}{N(q)}\Big ) \Big ), \\
 \xi_2(\gamma)=& \frac {C}{D \log^3 M}\frac {h(\gamma)g_1(\gamma)}{N(\gamma)H(\gamma)}\Big (  \log^2 \frac {M}{N(\gamma)}\log (\sqrt{X}N(\gamma))+\frac 2{3}\log^3 \frac {M}{N(\gamma)}+O\Big ( \log^2 M(1+ \sum_{\substack{q \equiv 1 \bmod {(1+i)^3} \\ q | \gamma}}\frac {\log N(q)}{N(q)}\Big ) \Big ), \\
 \xi_3(\gamma) \ll &  \frac {|h(\gamma)g_1(\gamma)|}{N(\gamma)H(\gamma)}\lp 1+ \sum_{\substack{q \equiv 1 \bmod {(1+i)^3} \\ q | \gamma}}\frac {\log N(q)}{N(q)} \rp.
\end{split}
\end{align*}

  We now expand $\log^3 (X/N(ab))$ in terms of $\log X, \log N(a)$ and $\log N(b)$ to recast
\begin{align}
\label{secm:1stterm}
 & \frac {\pi^4}{4^3} \frac{ D\widehat{\Phi}(1)}{36\zeta_{K}(2)}\sum_{\substack{\gamma \equiv 1 \bmod {(1+i)^3} }}\frac {N(\gamma)H(\gamma)}{h(\gamma)}\sum_{\substack{a,b \equiv 1 \bmod {(1+i)^3} \\ (a,b)=1}}\frac {\lambda(a\alpha)}{h(a)}\frac {\lambda(b\alpha)}{h(b)}\frac {ad_{[i]}(a)}{\sigma_{[i]}(a)} \frac {b d_{[i]}(a)}{\sigma_{[i]}(b)} \log^3\frac {X}{N(ab)}
\end{align}
  as a linear combination of terms
\begin{align*}
 & \frac {\pi^4}{4^3} \frac{ D\widehat{\Phi}(1)}{36\zeta_{K}(2)}\sum_{\substack{\gamma \equiv 1 \bmod {(1+i)^3} }}\frac {N(\gamma)H(\gamma)}{h(\gamma)}\xi_j(\gamma)\xi_k(\gamma)\log^l X,
\end{align*}
  where $j+k+l=3$.

  We can evaluate these terms using the expressions for $\xi_i(\gamma), 1 \leq i \leq 3$. Then applying \eqref{elemargm}
and partial summation, we see that
\begin{align}
\label{1stermest}
  \eqref{secm:1stterm}  \sim \left ( \frac \pi{4} \right )^4 \left (\frac 2{81}+\frac {28}{135\theta}+\frac {11}{18\theta^2}+\frac {70}{81\theta^3}+\frac {16}{27\theta^4}+\frac {4}{27\theta^5} \right ) \frac{ 2\pi \widehat{\Phi}(1)}{3\zeta_{K}(2)}.
\end{align}

  This treats one of the terms given in \eqref{secmoment}. To treat the other terms, we proceed similarly to the treatments done on \cite[p. 485]{sound1} to see that for primary, square-free $\gamma$ such that $N(\gamma) \leq M$,
\begin{align*}
 &  \sum_{\substack{a \equiv 1 \bmod {(1+i)^3} }}\frac {\lambda(a\alpha)}{h(a)}\frac {ad_{[i]}(a)}{\sigma_{[i]}(a)}  \Big ( \sum_{\substack{\varpi \equiv 1 \bmod {(1+i)^3} \\ \varpi |a}}\log^2 N(\varpi) \Big ) \\
=& - \frac {C}{D \log^3 M}\frac {h(\gamma)g_1(\gamma)}{N(\gamma)H(\gamma)}\lp \log^2 \frac {M}{N(\gamma)}\log (\sqrt{X}N(\gamma))+\frac 23\log^3\frac {M}{N(\gamma)}+O(\log^2 X) \rp,
\end{align*}
   and that
\begin{align*}
 &  \sum_{\substack{a \equiv 1 \bmod {(1+i)^3} }}\frac {\lambda(a\alpha)}{h(a)}\frac {ad_{[i]}(a)}{\sigma_{[i]}(a)} \log N(a) \Big ( \sum_{\substack{\varpi \equiv 1 \bmod {(1+i)^3} \\ \varpi |a}}\log^2 N(\varpi) \Big )
 \ll   \frac {|h(\gamma)g_1(\gamma)|}{N(\gamma)H(\gamma)}\Big ( 1+ \sum_{\substack{q \equiv 1 \bmod {(1+i)^3} \\ q | \gamma}}\frac {\log N(q)}{N(q)} \Big ).
\end{align*}

   As consequences, we see that
\begin{align*}
 & -\left ( \frac \pi{4} \right )^4 \frac{ D\widehat{\Phi}(1)}{36\zeta_{K}(2)}\sum_{\substack{\gamma \equiv 1 \bmod {(1+i)^3} }}\frac {N(\gamma)H(\gamma)}{h(\gamma)}\sum_{\substack{a,b \equiv 1 \bmod {(1+i)^3} \\ (a,b)=1}}\frac {\lambda(a\alpha)}{h(a)}\frac {\lambda(b\alpha)}{h(b)}\frac {ad_{[i]}(a)}{\sigma_{[i]}(a)} \frac {b d_{[i]}(a)}{\sigma_{[i]}(b)} \\
& \times  \log \frac {X}{N(ab)}\Big ( \sum_{\substack{\varpi \equiv 1 \bmod {(1+i)^3} \\ \varpi |a}}\log^2 N(\varpi)+ \sum_{\substack{\varpi \equiv 1 \bmod {(1+i)^3} \\ \varpi |b}}\log^2 N(\varpi) \Big  ) \\
\sim & \left ( \frac \pi{4} \right )^4 \left (\frac 2{81}+\frac {4}{45\theta}+\frac {7}{54\theta^2}+\frac {2}{27\theta^3} \right ) \frac{ 2\pi \widehat{\Phi}(1)X}{3\zeta_{K}(2)}.
\end{align*}

   Combining the above with \eqref{1stermest},  we find that the second mollified moment is
\begin{align}
\label{secondmoment}
  \sim \left ( \frac \pi{4} \right )^4 \left (\frac 4{81}+\frac {8}{27\theta}+\frac {20}{27\theta^2}+\frac {76}{81\theta^3}+\frac {16}{27\theta^4}+\frac {4}{27\theta^5} \right ) \frac{ 2\pi \widehat{\Phi}(1)}{3\zeta_{K}(2)}.
\end{align}

    Applying Cauchy-Schwarz inequality together with the first mollified moment \eqref{1stmollifiedmom} and the second mollified
moment \eqref{secondmoment}, we have
\begin{align}
\label{comparison}
\begin{split}
 & \sum_{\substack{X \leq N(d) \leq 2X \\ (d, 2)=1 \\ L(\half, \chi_{(1+i)^5d}) \neq 0}}\mu_{[i]}(d)^2 \geq
\sum_{\substack{(d, 2)=1 \\ L(\half, \chi_{(1+i)^5d}) \neq 0}}\mu_{[i]}(d)^2\Phi(\frac {N(d)}{X}) \geq
X \frac {S(L(\half, \chi_{(1+i)^5d})M(d); \Phi)^2}{S(L(\half, \chi_{(1+i)^5d})^2M^2(d); \Phi)} \\
\geq & \left (1-\frac 1{(\theta+1)^3} \right )\frac {2\pi}{3\zeta_K(2)}X=\lp \frac 78+o(1) \rp \sum_{\substack{X \leq N(d) \leq 2X \\ (d, 2)=1 }}\mu_{[i]}(d)^2,
\end{split}
\end{align}
 since we have that (see \cite[Section 3.1]{G&Zhao4})
\begin{align*}
 & \sum_{\substack{X \leq N(d) \leq 2X \\ (d, 2)=1 }}\mu_{[i]}(d)^2 \sim \frac {2\pi}{3\zeta_K(2)}X.
\end{align*}
  We now set $\theta=1-\varepsilon$ in \eqref{comparison} to see that the assertion of Theorem \ref{thm: nonvanishing} follows by summing over $X=x/2^j$ for $j \geq 1$ and this completes the proof.

\vspace*{.5cm}

\noindent{\bf Acknowledgments.} P. G. is supported in part by NSFC grant 11871082.

\bibliography{biblio}
\bibliographystyle{amsxport}

\vspace*{.5cm}

\end{document}